\documentclass[letterpaper]{amsart}

\usepackage{cmap}             
\usepackage[T1]{fontenc}      
\usepackage[english]{babel}   
\usepackage{microtype}        
\usepackage{csquotes}         
\usepackage[shortlabels]{enumitem}         
\usepackage{multirow}         
\usepackage{booktabs}         
\usepackage{array}            
\usepackage{xcolor}
\usepackage{newtxtext}        
\usepackage{comment}

\definecolor{black}{HTML}{000000}
\definecolor{OIOrange}{HTML}{E69F00}
\definecolor{OISkyBlue}{HTML}{56B4E9}
\definecolor{OIBluishGreen}{HTML}{009E73}
\definecolor{OIYellow}{HTML}{F0E442}
\definecolor{OIBlue}{HTML}{0072B2}
\definecolor{OIVermilion}{HTML}{D55E00}
\definecolor{OIPurple}{HTML}{CC79A7}

\usepackage{mathtools}        
\usepackage{amssymb}          
\usepackage{amsthm}
\usepackage{newtxmath}

\usepackage{thmtools} 		    
\usepackage{graphicx}         
\usepackage{subcaption}       
\usepackage{tikz}
\usetikzlibrary{arrows, decorations.markings, arrows.meta, shapes.misc}

\usepackage{import}
\usepackage[
  backend=biber,
  giveninits=true,
  isbn=false,
  url=false,
  style=alphabetic,
  sorting=nyt]{biblatex}
\renewbibmacro*{in:}{%
  \ifentrytype{article}
    {}
    {\printtext{\bibstring{in}\intitlepunct}}
}

\usepackage{xurl}           
\usepackage{hyperref} 
\hypersetup{
    colorlinks=true,
    linkcolor=OIVermilion,     
    citecolor=OIBlue,          
    urlcolor=OISkyBlue,        
    filecolor=OIBluishGreen,   
    pdfauthor={Mauro Artigiani, Ángel Pardo},
    pdftitle={Counting Z-covers of square-tiled surfaces},
    pdfsubject={Mathematics},
    pdfcreator={pdfLaTeX}
}
\usepackage[noabbrev, nameinlink]{cleveref} 

\declaretheorem[name=Theorem, numberwithin=section]{theorem}
\declaretheorem[name=Lemma, sibling=theorem]{lemma}
\declaretheorem[name=Proposition, sibling=theorem]{proposition}
\declaretheorem[name=Corollary, sibling=theorem]{corollary}

\declaretheorem[
  name=Theorem,
  Refname={Theorem,Theorems},
  refname={theorem,theorems}
]{bigthm}

\declaretheorem[name=Definition, sibling=theorem, style=definition]{definition}

\declaretheorem[name=Remark, sibling=theorem, style=remark, qed=$\triangle$]{remark}

\DeclarePairedDelimiter{\pairing}{\langle}{\rangle}

\providecommand{\given}{}
\newcommand{\SetSymbol}[1][]{%
	\nonscript\:\mathord{:}
	\allowbreak
	\nonscript\:
	\mathopen{}}
\DeclarePairedDelimiterX{\Set}[1]\{\}{%
	\renewcommand{\given}{\SetSymbol[\delimsize]}
	#1
}

\DeclarePairedDelimiterX\intcc[2]{[}{]}{#1, #2}
\DeclarePairedDelimiterX\intoo[2]{(}{)}{#1, #2}
\DeclarePairedDelimiterX\intoc[2]{(}{]}{#1, #2}
\DeclarePairedDelimiterX\intco[2]{[}{)}{#1, #2}

\newcommand{\numberset}{\mathbf}
\newcommand{\NN}{\numberset{N}}
\newcommand{\ZZ}{\numberset{Z}}
\newcommand{\QQ}{\numberset{Q}}
\newcommand{\RR}{\numberset{R}}
\newcommand{\CC}{\numberset{C}}

\newcommand{\TT}{\numberset{T}}
\renewcommand{\epsilon}{\varepsilon}

\newcommand{\cC}{\mathcal{C}}
\newcommand{\cH}{\mathcal{H}}
\newcommand{\cM}{\mathcal{M}}

\newcommand{\sC}{\mathscr{C}}

\DeclareMathOperator{\SL}{SL}
\DeclareMathOperator{\hol}{hol}
\DeclareMathOperator{\ann}{Ann}
\DeclareMathOperator{\proj}{pr}
\DeclareMathOperator{\cw}{cw}
\DeclareMathOperator{\vol}{vol}
\DeclareMathOperator{\area}{area}
\DeclareMathOperator{\Aff}{Aff}

\newcommand{\slz}{{\SL_2(\ZZ)}}
\newcommand{\slr}{{\SL_2(\RR)}}

\newcommand{\tS}{\widetilde{S}}
\newcommand{\qS}{\widetilde{S}^q}

\newcommand{\modn}[1]{\mathbin{\equiv_{#1}}}
\newcommand{\nmodn}[1]{\mathbin{\not\equiv_{#1}}}

\tikzset{
    k-bull/.style={circle, fill=black, inner sep=1.5pt},
    v-bull/.style={circle, fill=OIVermilion, inner sep=1.5pt},      
    b-bull/.style={circle, fill=OISkyBlue, inner sep=1.5pt},        
    g-bull/.style={circle, fill=OIBluishGreen, inner sep=1.5pt},    
    bb-bull/.style={circle, fill=OIBlue, inner sep=1.5pt},          
    o-bull/.style={circle, fill=OIOrange, inner sep=1.5pt}          
}

\begin{document}


\title[Counting cylinders on covers of genus $2$ square-tiles]{Counting cylinders on $\ZZ$-covers of \\ genus $2$ square-tiled surfaces}
\author[M. Artigiani]{Mauro Artigiani}
\address{
Departamento de Matemáticas,
Universidad Nacional de Colombia,
Ak 30, \#45-03, Bogotá, Colombia.
}
\email{martigiani@unal.edu.co}

\author[A. Pardo]{Angel Pardo}
\address{
Departamento de Matemáticas,
Universidad de Chile,
Las Palmeras 3425, Ñuñoa, Santiago, Chile.
}
\email{angelpardo@uchile.cl}

\date{\today}


\begin{abstract}
  We count maximal cylinders on zero holonomy $\ZZ$-covers of genus $2$
  square-tiled surfaces, up to $\ZZ$-action, obtaining quadratic asymptotics. We
  also show that the leading term of the asymptotic, called the Siegel--Veech
  constant, can be recovered via a large-genus approximation by intermediate
  finite covers. Our work applies to the infinite staircases introduced by
  P.~Hubert and G.~Weitze-Schmithüsen. For many members of this family, we
  explicitly compute the associated Siegel--Veech constants. In particular, we
  exhibit the first infinite family of examples of zero holonomy $\ZZ$-cover in
  which the number of cylinders grows sub-quadratically.
\end{abstract}

\maketitle

\section{Introduction}
The classical Gauss circle problem states that the number $N(L)$ of (families of)
parallel simple closed geodesics on the flat torus $\TT^2 = \RR^2/\ZZ^2$ of
length at most $L$ grows quadratically in $L$ and, more precisely,
\begin{equation}\label{eq:Gauss}    
    N(L) \sim \frac{1}{2\zeta(2)} \cdot \pi L^2,
\end{equation}
where $\zeta$ is the Riemann zeta function and the factor $1/2$ comes from the
fact that we are counting unoriented geodesics. The constant $1/2\zeta(2)$ is
called the \emph{Siegel--Veech constant} of this counting problem\footnote{We
remark that, by convention, it is the constant of $\pi L^2$ and not the one of
the asymptotic.}.

Compact translation surfaces are higher genus generalization of flat tori. They
carry a singular Euclidean structure, with finitely many conical singularities
of cone angle a positive integral multiple of $2\pi$. In particular, geodesics
for the flat metric are given by following straight lines in charts. Similarly
to the case of the flat torus, closed geodesics come in families of parallel
lines, which fill a cylinder bounded by \emph{saddle connections}, that is
geodesics segments joining two (not necessarily distinct) singularities. 

In the case of the flat torus, if we mark the origin and create an
\enquote{artificial singularity}, we could also count saddle connections
connecting this singularity to itself. This count clearly corresponds to the one
for cylinders in \cref{eq:Gauss}. In the case of higher genus translation
surfaces, however, these are two different problems.

It is an open problem to determine whether the asymptotic for counting families
of saddle connections or (maximal) cylinders of simple closed geodesics of length
at most $L$, for \emph{all} translation surfaces is quadratic.

This article deals with counting on $\ZZ$-covers of square-tiled translation
surfaces. More precisely, we study the asymptotic of cylinders of simple closed
geodesics of length at most $L$ for zero holonomy $\ZZ$-covers of genus $2$
square-tiled surfaces, up to the natural $\ZZ$-action. In particular, our
results apply to one of the first examples of such surfaces: the
Hubert--Weitze-Schmithüsen staircases~\cite{Hubert-Schmithusen}, see
\Cref{sec:HSstaircases}. 

We show that the counting obeys a quadratic asymptotic, and, in many cases,
obtain explicitly the \emph{Siegel--Veech constant} which gives the leading
coefficient of the asymptotic. Moreover, we show that this constant can be
recovered by studying the corresponding counting problem on $\ZZ_q$-covers, for
a prime $q$, as $q\to\infty$. In other words, we obtain the Siegel--Veech
constant as a \emph{high-genus limit} of related Siegel--Veech constants.

\subsection{Main results}
A \emph{square-tiled} (translation) surface (also known as \enquote{origami}) is
a pair $(S,\omega)$ where $S$ is a translation cover of the flat torus $\TT^2$,
ramified only over the origin and $\omega$ is the pull-back under the covering
map of the $1$-form induced on $\TT^2$ by $dz$, see \Cref{fig:case00} for an
example and \Cref{sec:square-tiled_surfs} for more details.

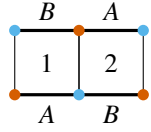
\begin{figure}[tbh]
    \centering
    \begin{tikzpicture}[scale=.85]
        \foreach \i in {0,...,1}
            \draw (\i,0) rectangle ++(1,1);

        \node at (0.5,0.5) {$1$};
        \node at (1.5,0.5) {$2$};
	
        \draw[very thick] (0,0) -- ++(1,0) node[midway,below] {$A$};
        \draw[very thick] (1,1) -- ++(1,0) node[midway,above] {$A$};
        \draw[very thick] (1,1) -- ++(-1,0) node[midway,above] {$B$};
        \draw[very thick] (2,0) -- ++(-1,0) node[midway,below] {$B$};
        
        \node[v-bull] at (0,0) {};
        \node[v-bull] at (2,0) {};
        \node[v-bull] at (1,1) {};
        \node[b-bull] at (0,1) {};
        \node[b-bull] at (2,1) {};
        \node[b-bull] at (1,0) {};
    \end{tikzpicture}
    \caption{$S_{2,0}\in\cH(0,0)$. In this case $\sigma_h = \sigma_v = (1 \, 2)$.}
    \label{fig:case00}
\end{figure}

Let $S$ be a square-tiled surface of genus $2$, and consider a $\ZZ$-cover $\tS$
of $S$. As we explain in \Cref{sec:Zcovers}, such a cover is determined by a
single homology class $\eta \in H_1(S,\Sigma; \ZZ)$. We ask that the holonomy of
$\eta$ is $0$, since these are the only dynamically relevant covers. We refer to
such a cover as a \enquote{zero holonomy cover}, see \Cref{sec:Zcovers} for more
details. Let $N(\tS, L)$ denote the number of maximal cylinders of closed
geodesics of length at most $L$, up to translation by $\ZZ$. Equivalently,
$N(\tS,L)$ counts the number of isotopy classes of simple non-oriented closed
geodesics of length at most $L$, up to $\ZZ$-action. Our first main result is
the precise asymptotic for $N(\tS, L)$.

\begin{bigthm}\label{bigthm:asymptotic}
    Let $S$ be a square-tiled surface of genus $2$, and consider a zero
    holonomy cover $\ZZ$-cover $\tS$ of $S$, branched only over the
    singularities of $S$. Then, there exists a non negative constant $c(\tS)$ such that 
    \[
        N(\tS,L) = c(\tS) \cdot \pi L^2 + o(L^2).
    \]
\end{bigthm}

We remark that, very recently, in~\cite{ACMSS}, a similar result was obtained
for \emph{finite} branched cyclic covers of \emph{generic} translation surfaces.
Since square-tiled surfaces are \emph{not} generic, our result can be seen as a
counterpart to some of their results.

The main difficulty in \Cref{bigthm:asymptotic} is the presence of relative
homology, which makes the topological analysis more delicate, see
\Cref{sec:sketch} for more details.

The constant $c(\tS)$ is called the \emph{Siegel--Veech} constant associated to
the counting problem. This kind of constants appear in the Siegel--Veech
formula~\cite[Theorem~2.2]{Eskin-Masur}, a Siegel-type formula introduced by
Veech in~\cite{Veech98}; see \Cref{sec:history} for references and \Cref{sec:SV}
for its definition. 

An important class of examples that our \Cref{bigthm:asymptotic} applies to is
the class of infinite staircases, defined by P.~Hubert and G.~Weitze-Schmithüsen
in~\cite{Hubert-Schmithusen}, see \Cref{sec:HSstaircases} for their definition.
We explicitly obtain the Siegel--Veech constants for many
Hubert--Weitze-Schmithüsen staircases in
\Cref{cor:SV_staircases_H2,cor:SV_staircases_H11}, and
\Cref{cor:SV_staircases_H00}.

In particular, as a consequence of these computations we exhibit an infinite
family of examples of zero holonomy $\ZZ$-covers in which the Siegel--Veech
constant of \Cref{bigthm:asymptotic} is actually $0$. In other words, the number
of cylinders grows sub-quadratically.

\begin{bigthm}\label{bigthm:subquadratically} 
    There exists an infinite family of zero holonomy $\ZZ$-covers of genus $2$
    square-tiled surfaces for which the number of cylinders grows
    \emph{sub-quadratically} with respect to the length.
\end{bigthm}

To the best of our knowledge, this is the first such example.

Our last result is about large-genus asymptotics of Siegel--Veech constants. Let
$q$ be a prime number and consider $\qS$, a cyclic cover of degree $q$ of a
genus $2$ square-tiled surface $S$. For each fixed $q$, it is known by the work
of Veech~\cite{Veech89} that $N(\qS,L)\sim c(\qS)\pi L^2$. We show that the
Siegel--Veech constant $c(\tS)$ of the $\ZZ$-cover can be recovered as a limit
of Siegel--Veech constants as $q$ (and hence the genus of $\qS$) goes to
infinity.

\begin{bigthm}\label{bigthm:large_genus}
    Let $S$ be a genus two square-tiled surface and let $\eta\in
    H_1(S,\Sigma;\ZZ)$ be a zero holonomy homology class. Let $q$ be a prime
    number and $\qS$ the $\ZZ_q$-cover of $S$ determined by $\eta$. Then
    \[
        \frac{1}{q} c(\qS) \xrightarrow{q\to\infty} c(\tS),
    \]
    where $c(\qS)$ is the Siegel--Veech constant associated to $N(\qS,L)$ and
    $\tS$ is the $\ZZ$-cover of $S$ determined by $\eta$.
\end{bigthm}

The factor $1/q$ in the left hand side comes from modding out the action of $\ZZ_q$, which is mandatory for the $\ZZ$-cover in the right hand side of the limit.

\subsection{Related results}\label{sec:history}
Siegel--Veech constants are related to important problems such as the
computations of Masur--Veech volumes of strata of normalized translation
surfaces~\cite{Eskin-Masur-Zorich} and the sum of Lyapunov exponents for the
Teichmüller geodesic flow~\cite{Eskin-Kontsevich-Zorich}.

As we mentioned above, for a general translation surface, counting cylinders and
saddle connections are different problems, even though both growth rates are
believed to be quadratic. Masur proved in~\cite{Masur88,Masur90} that, for
\emph{every} translation surface $S$, there exist constants $c$ and $C$ such
that, for large enough $L$,
\[
    c \pi L^2 \le N_\text{cyl} (S,L) \le N_\text{sc} (S,L) \le C\pi L^2,
\]
where $N_\text{cyl}$ and $N_\text{sc}$ correspond to counting cylinders and
saddle connections respectively. It is still an open question whether this
result can be made into an asymptotic.

Veech proved in~\cite{Veech89} that there is a quadratic asymptotic for Veech
surfaces (such as square-tiled surfaces). A different proof of this was given by
Gutkin and Judge in~\cite{Gutkin-Judge}.

Other, different, proofs of the quadratic upper bound for both the cylinders and
saddle connections counting problems were given also by Vorobets~\cite{Vorobets}
and Eskin and Masur~\cite{Eskin-Masur}. Moreover, in~\cite{Eskin-Masur}, it was
shown that there exist constants $c_\text{cyl}$ and $c_\text{sc}$ such that, for
\emph{almost all} surfaces $S$, $N_\text{cyl} (S,L)\sim c_\text{cyl} \pi L^2$
and $N_\text{sc} (S,L)\sim c_\text{sc} \pi L^2$. Here, \enquote{almost all} is
with respect to any $\slr$-invariant probability ergodic measure, such as the
Masur--Veech measure, with the constants $c_\text{cyl}$ and $c_\text{sc}$
depending on the chosen measure.

Siegel--Veech constants were computed for Veech surfaces
in~\cite{Veech89,Vorobets,Gutkin-Judge,Schmoll} and for some non Veech surfaces
in~\cite{Eskin-Masur-Schmoll,Eskin-Marklof-Morris}. The generic ones
corresponding to connected components of strata of abelian differentials were
computed in~\cite{Eskin-Masur-Zorich}. In general, Siegel--Veech constants for
Veech surfaces are not the same one of the strata the surfaces live in.

The second named author studied the counting problem for the Ehrenfest wind-tree
model (and the Delecroix--Zorich generalization~\cite{Delecroix-Zorich})
in~\cite{Pardo:counting}, and give an effective version in~\cite{Pardo:error}.
The solution is linked to a \emph{non-varying phenomenon} as explained
in~\cite{Pardo:nonvarying}.

Recently, there has been an interest in studying large-genus asymptotics of
Siegel--Veech constants, and we see our \Cref{bigthm:large_genus} as a result
along this line of results. For more details, we refer the reader
to~\cite{Eskin-Zorich,Chen-Moller-Zagier,CMSZ,Aggarwal19,Aggarwal20,Aggarwal21}.
As we mentioned already, in~\cite{ACMSS}, the authors study a related problem to
the one we study, and obtain several deep results on Siegel--Veech constants for
finite branched cyclic covers of generic surfaces.

\subsection{A sketch of the proofs}\label{sec:sketch} 

We are concerned with counting (maximal cylinders of) closed geodesics in
$\ZZ$-covers of genus~$2$ square-tiled surfaces. Any closed geodesic can be
$\ZZ$-translated to obtain a new (non-isotopic) closed geodesic. Thus, we count
(maximal cylinders of) closed geodesics of bounded length in the cover $\tS$, up
to $\ZZ$-translations. 

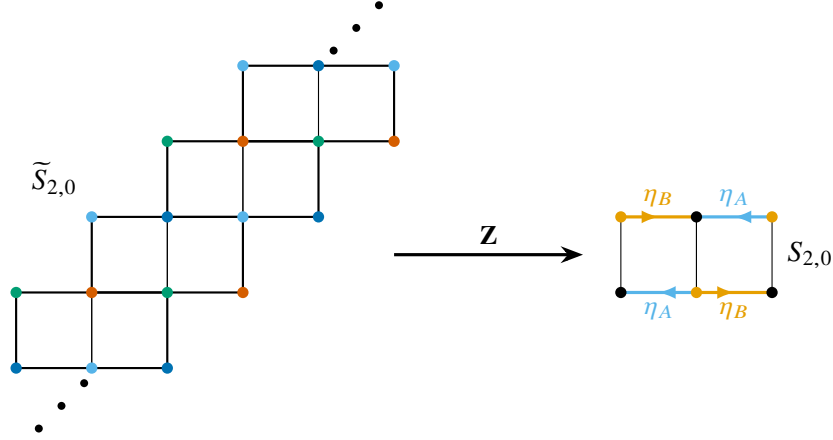
\begin{figure}[tbhp]
    \centering
    \begin{tikzpicture}[mid arrow/.style={
			   postaction={decorate,decoration={
				   markings,
				   mark=at position .5 with {\arrow{latex}}
		   }}}]

        
        
        \foreach \level in {0, 1, 2, 3} {
            \pgfmathsetmacro{\y}{\level}
            \pgfmathsetmacro{\x}{1+\level}
            
            \draw (\x, \y) rectangle ++(1,1);
            \draw (\x+1, \y) rectangle ++(1,1);
            
            \draw[thick] (\x, \y) rectangle ++(2,1);
        }

        \fill (5.2, 4.2) circle (1.5pt);
        \fill (5.5, 4.5) circle (1.5pt);
        \fill (5.8, 4.8) circle (1.5pt);

        \fill (1.9, -0.2) circle (1.5pt);
        \fill (1.6, -0.5) circle (1.5pt);
        \fill (1.3, -0.8) circle (1.5pt);

        \node[bb-bull] at (1,0) {};
        \node[bb-bull] at (3,0) {};
        \node[bb-bull] at (3,2) {};
        \node[bb-bull] at (5,2) {};
        \node[bb-bull] at (5,4) {};
        \node[v-bull] at (2,1) {};
        \node[v-bull] at (4,1) {};
        \node[v-bull] at (4,3) {};
        \node[v-bull] at (6,3) {};
        \node[b-bull] at (2,0) {};
        \node[b-bull] at (2,2) {};
        \node[b-bull] at (4,2) {};
        \node[b-bull] at (4,4) {};
        \node[b-bull] at (6,4) {};
        \node[g-bull] at (1,1) {};
        \node[g-bull] at (3,1) {};
        \node[g-bull] at (3,3) {};
        \node[g-bull] at (5,3) {};
        
        \draw[-Stealth, very thick] (6,1.5) -- (8.5,1.5) node[midway,above] {$\ZZ$};

            \foreach \i in {9,...,10}
                \draw (\i,1) rectangle ++(1,1);

        \draw[OISkyBlue, mid arrow, very thick] (10,1) -- (9,1) node[midway,below] {$\eta_A$};
        \draw[OISkyBlue, mid arrow, very thick] (11,2) -- (10,2) node[midway,above] {$\eta_A$};
        \draw[OIOrange, mid arrow, very thick] (9,2) -- (10,2) node[midway,above] {$\eta_B$};
        \draw[OIOrange, mid arrow, very thick] (10,1) -- (11,1) node[midway,below] {$\eta_B$};
        
        \node[k-bull] at (9,1) {};
        \node[o-bull] at (10,1) {};
        \node[k-bull] at (11,1) {};
        \node[o-bull] at (9,2) {};
        \node[k-bull] at (10,2) {};
        \node[o-bull] at (11,2) {};

        \node at (11.5,1.5) {{\large $S_{2,0}$}};
        \node at (1.5,2.5) {{\large $\tS_{2,0}$}};
    \end{tikzpicture}
    \caption{The $\ZZ$-cover $\tS_{2,0}$ of $S_{2,0}$ defined by the zero holonomy
        curve $\eta = \eta_A + \eta_B$. The colored dots on the covering surface
        represent the four infinite conical angle singularities, whereas the two 
        on $S_{2,0}$ are marked points.}
    \label{fig:staircase_cover}
\end{figure}

Since the algebraic intersection form $\pairing{\cdot,\cdot}$ is non-degenerate
between $H_1(S, \Sigma; \ZZ)$ and $H_1(S \setminus \Sigma; \ZZ)$, every such
cover can be described by a primitive element $\eta$ in the group of relative
homology $H_1(S, \Sigma; \ZZ)$, see \Cref{fig:staircase_cover} for an example
and~\cite{Hooper-Weiss,DHV} and \Cref{sec:Zcovers} for details. We recall that
the \emph{holonomy} of a curve $\eta$ is given by developing $\eta$ in charts
and taking the difference of its endpoints. We denote this by $\hol(\eta)$. We
ask for the so-called \emph{no-drift condition}: $\hol(\eta) = 0$. This is
motivated by dynamical reasons, see,
e.g.,~\cite{Hooper-Weiss,Artigiani,Avila-Hubert}. In the literature,
$\ZZ$-covers defined by homology classes of zero holonomy are also called
\emph{recurrent $\ZZ$-covers}.

Cylinders in $\tS$ project down to cylinders in $S$. However, cylinders in $S$
lift to cylinders in $\tS$ if and only if the homology class of its core curve
has zero intersection with the relative cycle $\eta$ defining the cover.
Otherwise, they lift to infinite strips.

Let $K$ be smallest subspace of $H_1 (S, \Sigma; \ZZ)$ that contains $\eta$ and
that is invariant under affine transformations of $S$. The cylinders $C$ with
core curve $\gamma_C$ that have zero intersection with $\eta$ split naturally
into two families:
\begin{enumerate}
    \item if $\pairing{\gamma_C, \kappa} = 0$ for all $\kappa\in K$, we say that
    $C$ is a \emph{good cylinder}.
    \item if $\pairing{\gamma_C, \eta} = 0$, but $C$ is not good, we naturally
    call $C$ a \emph{bad cylinder}.
\end{enumerate}

Historically, the notion of good cylinders was first introduced by Avila and
Hubert in~\cite{Avila-Hubert} in order to give a geometric criterion for
recurrence of $\ZZ^d$-periodic translations surfaces. 

If the base surface is in $\cH(2)$, there is no relative homology and previous
work by the second named author in~\cite{Pardo:counting} implies that bad
cylinders are asymptotically negligible. Hence, one of the contributions of this
paper is to study the case when the base surface is in $\cH(1,1)$. In this case,
the presence of relative homology prevents us to use symplecticity as
in~\cite{Pardo:counting}, so we use a new strategy using the arithmeticity of
the Kontsevich--Zorich monodromies in genus $2$~\cite{Betal} and some ideas on
finite cyclic covers as in~\cite{ACMSS}.

Since bad cylinders are negligible, we are reduced to counting good cylinders.
As a consequence of works of Veech~\cite{Veech89,Veech92} and Eskin and
Masur~\cite{Eskin-Masur}, good cylinders have quadratic asymptotic. It remains
to compute the associated Siegel--Veech constants. For $S\in\cH(2)$, we exploit
previous work of Lelièvre~\cite{Lelievre}, and Lelièvre and
Royer~\cite{Lelievre-Royer}. The case $S\in\cH(1,1)$ is more delicate, and
relays on the \emph{parity condition}, see~\cite{Duryev}, hence we do not
compute them in full generality.

Let us now briefly explain how to prove \Cref{bigthm:large_genus}. We now
consider intermediate $\ZZ_q$-covers. In this case, we still have good and bad
cylinders as before, but we also need to consider cylinders $C$ whose core curve
$\gamma_C$ has a non-zero intersection, modulo $q$, with $\eta$. Geometrically,
these last ones lift on $\qS$ to a single cylinder which is $q$ times longer.
So, we call them \emph{strips modulo $q$}. The Siegel--Veech constant for
cylinders in $\qS$ is then obtained as a combination of the constants
corresponding to good and bad cylinders and to strips modulo $q$. However, in
this case the bad cylinders and the strips modulo $q$ are \emph{not} negligible.
Since we only care about the asymptotic as $q\to\infty$, we estimate the
\emph{not good} cylinders in terms of a Siegel--Veech constant of the base
surface $S$, and show that, with the appropriate normalizations, these cylinders
become negligible in the limit.

\subsection{Organization of the paper}
We collect all the preliminaries about translation and square-tiled surfaces and
their $\ZZ$-covers in \Cref{sec:preliminaries}. In particular, we describe the
Hubert--Weitze-Schmithüsen  staircases, briefly introduce Siegel--Veech
constants and explain how to compute them in our cases.

We deal with the case when the base surface is in $\cH(2)$ in \Cref{sec:H2}. We
then prove \Cref{bigthm:large_genus} for the case when the base surface is in
$\cH(2)$ in \Cref{sec:q-covers}. Then, we deal with the more delicate case of
$\cH(1,1)$, proving that bad cylinders are sub-quadratic, completing the proof
of \Cref{bigthm:large_genus} and proving \Cref{bigthm:subquadratically} in
\Cref{sec:H11}. Finally, to complete the study of the whole
Hubert--Weitze-Schmithüsen staircases family, we study the simpler genus~$1$
case, when the base surface is $S_{2,0}\in\cH(0,0)$ in \Cref{sec:H00}.

\subsection*{Acknowledgements}
We thank Casa Matemática Oaxaca and the organizers of the workshop
\enquote{Blooming Beasts: A Conference on the Geometry, Topology, and Dynamics
of Infinite-Type Surfaces} where this work began. A.P. was supported by ANID
through FONDECYT Regular grant 1261332. 

\section{Preliminaries}\label{sec:preliminaries}

\subsection{Translation surfaces}\label{sec:translation_surfs}

A compact translation surface is a pair $(S,\omega)$, where $S$ is a compact
Riemann surface and $\omega$ is a (not identically zero) holomorphic $1$-form,
usually called an abelian differential. Local integration of the $1$-form gives
an atlas of charts in which, away from the zeros of the form, the change of
coordinates are translations. General references for compact and non-compact
translation surfaces are, respectively,~\cite{AthreyaMasur, DHV}. 

Let $\Sigma\subseteq S$ be the set of zeros of $\omega$. Points in $\Sigma$
correspond conical singularities of conical angle $2(k_\sigma+1)\pi$, where
$k_\sigma\in\NN$ is the order of the zero of $\omega$ at the point
$\sigma\in\Sigma$. If the surface $S$ has genus $g$, we have a Gauss-Bonnet
formula:
\begin{equation}\label{eq:GaussBonnet}    
    \sum_{\sigma\in\Sigma} k_\sigma = 2g - 2.
\end{equation}

The moduli space of translation surfaces is naturally stratified by $\cH(k_1,
\dots, k_n) = \cH(\underline{k})$, where $k_i\in\NN$ satisfy
\cref{eq:GaussBonnet} above. Strata are not necessarily connected, however they
have at most three connected components, and their components have been
completely classified by Kontsevich and Zorich~\cite{KontsevichZorich}. In
particular, the strata we will consider ($\cH(0,0)$, $\cH(1,1)$ and $\cH(2)$)
are all connected.

Let $\cH_1(\underline{k})$ the ``unit hyperboloid'' formed by surfaces of area
$1$ inside $\cH(\underline{k})$. There is a natural action by the group $\slr$
on translation surfaces, given by post-composition in charts by matrices inside
$\slr$. This action preserves each connected component of the strata
$\cH_1(k_1,\dotsc,k_n)$. The \emph{Veech group} of a surface $(S,\omega)$ is the
stabilizer under this action. We will denote this group by $\Gamma(S)$.
Equivalently, let $\Aff(S)$ be the group of orientation preserving affine diffeomorphisms of $S$, that is the homeomorphisms which fix pointwise the singularities in $\Sigma$ and that are affine in the flat coordinates. If $D\colon\Aff(S)\to\slr$ denotes the derivative, then $\Gamma(S) = D(\Aff(S))$.

There is a natural invariant ergodic measure for the action of $\slr$ on
$\cH(\underline{k})$, which is called the \emph{Masur--Veech measure}. It is in
the Lebesgue class using the natural orbifold structure of $\cH(\underline{k})$
defined by period coordinates. If we restrict to surfaces of a fixed area, then
this measure is finite.

A \emph{saddle connection} is a geodesic connecting two, not necessarily
distinct, singularities of $(S,\omega)$. To a saddle connection $\gamma$ we can
associate its holonomy vector $\hol(\gamma)\in\CC$ by developing the geodesic in
the plane and taking the difference of the endpoints. Equivalently,
\[
    \hol(\gamma) =  \int_\gamma \omega.
\]
Holonomy vectors form a discrete subset of $\CC$, which is usually not a
lattice, see~\cite{Hubert-Schmidt:Veech}.

A \emph{cylinder} on a translation surface is an isometrically embedded copy of
$\RR/w\ZZ\times\intcc{0}{h}$ for some $h$ and $w$ in $\RR$, called respectively
the height and width of the cylinder. A maximal cylinder is bounded above and
below by saddle connections. Cylinders are foliated by closed geodesics,
isotopic to each other. Each of them is called a core curve of the cylinder.
Conversely, given a closed geodesic, we can enlarge it to obtain a maximal
cylinder of which the geodesic is a core curve.

\subsection{Square-tiled surfaces}\label{sec:square-tiled_surfs}
A translation surface $(S,\omega)$ is a \emph{square-tiled surface} if it is a
cover of the square torus, ramified only over the origin, and the cover is a
translation in the local coordinates. Equivalently, if $\pi\colon S \to\TT^2$ is
the covering map, we have that $\pi^*(\omega_0) = \omega$, where $\omega_0$ is
the $1$-form induced by $dz$ on $\TT^2$. Since the $1$-form is determined by the
covering, we will not write it.

A square-tiled surface $S$ is \emph{primitive} if there is no intermediate cover
between $S$ and $\TT^2$. Equivalently, the holonomy vectors span exactly
$\ZZ^2$. 

A \emph{square} of $S$ is a connected component of $\pi^{-1}(\intoo{0}{1}^2)$.
The number of squares of $S$ is exactly the degree $d$ of the covering. By
enumerating in any way the squares, we obtain two permutations of the set
$\Set{1,\dotsc, d}$: $\sigma_h$ and $\sigma_v$. By definition, $\sigma_h(i) = j$
(resp.\ $\sigma_v(i) = j$) if the right (resp.\ top) side of the $i^\text{th}$
square is identified to the left (resp.\ bottom) one of the $j^\text{th}$ square
(using the orientation induced by the covering), see \Cref{fig:case00} on
\cpageref{fig:case00} and the rest of the figures in this Section. We remark
that one can recover the square-tiled surface simply by the permutations.

\subsection{\texorpdfstring{$\ZZ$}{Z}-covers of translation
surfaces}\label{sec:Zcovers} Let $S$ be a compact translation surface, and
$\Sigma$ the set of its conical singularities. A $\ZZ$-cover of the surface
$S\setminus\Sigma$ is defined by the kernel of a surjective homomorphism
$\phi\colon\pi_1(S\setminus\Sigma)\to\ZZ$. Since $\ZZ$ is abelian, the map
$\phi$ factors through the first homology group $H_1(S\setminus\Sigma;\ZZ)$ and
can be represented by a cohomology class in $H^1(S\setminus\Sigma;\ZZ)$. Thanks
to the fact that the algebraic intersection form 
\[
    \pairing{\cdot,\cdot}\colon H_1(S\setminus\Sigma;\ZZ)\times H_1(S, \Sigma;\ZZ) \to\ZZ
\]
is non degenerate, there exists a non-zero curve $\eta\in H_1(S,\Sigma;\ZZ)$
which defines the covering, in the sense that $\phi$ map corresponds to the
homomorphism
\[
\begin{split}
    \phi_\eta \colon \pi_1(S\setminus\Sigma) &\to \ZZ\\
                    \gamma &\mapsto \pairing{[\gamma], \eta},
\end{split}
\]
where $[\gamma]$ is the class of $\gamma$ inside $H_1(S\setminus\Sigma;\ZZ)$.
For more details, see~\cite[Section~2.2]{DHV} or~\cite[Section~3]{Hooper-Weiss}.

We will always assume the \emph{no-drift} condition, meaning that the curve
$\eta$ satisfies $\hol(\gamma)=0$. In the context of $\ZZ$-covers, this class of
covers are also called recurrent $\ZZ$-covers, see~\cite{Hooper-Weiss}. We will
denote the zero holonomy part of the homology by $H_1^{(0)} (S,\Sigma;\ZZ)$.

We remark that, if the surface $S$ has genus $2$, then
\begin{equation}\label{eq:homology_splitting}
    H_1(S, \Sigma;\ZZ) = H_1^{\text{st}}(S;\ZZ) \oplus H_1^{(0)}(S, \Sigma;\ZZ),
\end{equation}
where $H_1^{\text{st}}(S;\ZZ)$ is the stable part, also called the
\emph{tautological plane}, generated by the duals of the real and imaginary part
of the abelian differential $\omega$ defining the square-tiled surface. A
similar splitting also holds for the absolute homology.

The above description implies that a cylinder on the surface $S$ with core curve
$\gamma$ lifts to a cylinder in a $\ZZ$-cover defined by the curve $\eta$ if and
only if
\[
    \pairing{\gamma,\eta} = 0.
\]
Otherwise, it lifts to an \emph{infinite strip}, meaning an isometrically
embedded copy of $\RR\times\intoo{a}{b}$ for some $a,b\in\RR$.

\begin{definition}[Good and bad cylinders]
    Let $S$  be a compact square-tiled surface and let $\eta\in H_1^{(0)}
    (S,\Sigma;\ZZ) \setminus \{0\}$. Let $K$ be the smallest $\Gamma(S)$-invariant subspace of $H_1^{(0)}(S,\Sigma;\ZZ)$ that contains $\eta$. Then, cylinders $C$ with core curve $\gamma_C$ that have zero
    intersection with $\eta$ naturally split into two families:
    \begin{enumerate}
        \item if $\pairing{\gamma_C, \kappa} = 0$ for all $\kappa\in K$, that is
        $\gamma_C \in\ann(K)$, we say that $C$ is a \emph{good cylinder}.
        \item if $\pairing{\gamma_C, \eta} = 0$, but $C$ is not good, we naturally
        call $C$ a \emph{bad cylinder}.
    \end{enumerate}
\end{definition}

In the previous Definition, if $S\in\cH(2)$, then $K = H_1^{(0)}(S,\Sigma;\ZZ)$.
However, if $S\in\cH(1,1)$ it may be a strictly smaller subspace, see
\Cref{sec:H11} for more details.

We conclude this section by showing that, in most cases, a cylinder is good if
and only if the homology class of its core curve lies in the tautological plane.
In these cases, good cylinders correspond to one-cylinder decompositions of $S$.
In fact, we have the following.

\begin{lemma}\label{lemma:good_are_one-cyl} 
    Let $S$ be a square-tiled surface of genus $2$ and let $\eta$ be a zero
    holonomy curve. Assume that $K = H_1^{(0)}(S,\Sigma;\ZZ)$ is the smallest
    $\Gamma(S)$-invariant subspace of $H_1^{(0)}(S,\Sigma;\ZZ)$ containing
    $\eta$. Then, a cylinder is good if and only if it corresponds to a
    one-cylinder decomposition of $S$.
\end{lemma}

\begin{proof}
    We begin by recalling that every stratum in genus $2$ is connected and
    coincides with the hyperelliptic component. Hence, the hyperelliptic
    involution preserves each cylinder in a cylinder decomposition.

    Let us show that we cannot have two cylinders in $S$ whose core curve are
    homologous. Let $C_1$ and $C_2$ two such cylinders. Since $S$ is connected,
    the two cylinders are glued to each other either on the top or the bottom
    side. By the above remark, the hyperelliptic involution preserves the two
    cylinders. Hence, $C_1$ and $C_2$ are glue both along their top and bottom
    sides. Then, cutting along the core curves of the cylinders leaves the
    surface $S$ connected, contractading the fact that the core curves were
    homologous.
\end{proof}

\subsection{An example of \texorpdfstring{$\ZZ$}{Z}-covers: Hubert--Weitze-Schmithüsen  staircases}\label{sec:HSstaircases}

\begin{figure}[tbh]
    \centering
    \begin{tikzpicture}
        \foreach \i in {0,...,3}
            \draw (\i,0) rectangle ++(1,1);
        \foreach \i in {1,...,3}
            \draw (3,\i) rectangle ++(1,1);

        \node at (0.5,0.5) {$1$};
        \node at (1.5,0.5) {$\cdots$};
        \node at (2.5,0.5) {$\cdots$};
        \node at (3.5,0.5) {$a$};
        \node at (3.5,1.5) {$a+1$};
        \node at (3.5,2.5) {$\vdots$};
        \node at (3.5,3.5) {$a+b$};
	
        \draw[very thick] (0,0) -- ++(1,0) node[midway,below] {$A$};
        \draw[very thick] (3,4) -- ++(1,0) node[midway,above] {$A$};
        \draw[very thick] (1,1) -- ++(-1,0) node[midway,above] {$B$};
        \draw[very thick] (4,0) -- ++(-1,0) node[midway,below] {$B$};
        
        \node[v-bull] at (0,0) {}; 
        \node[v-bull] at (4,0) {};
        \node[v-bull] at (1,1) {};
        \node[v-bull] at (1,0) {};
        \node[v-bull] at (4,4) {};
        \node[v-bull] at (3,4) {};

        \node[b-bull] at (0,1) {};
        \node[b-bull] at (4,1) {};
        \node[b-bull] at (3,1) {};
        \node[b-bull] at (3,0) {};
    \end{tikzpicture}
    \caption{The square-tiled surface $S_{a,b}\in\cH(1,1)$, an example of
    Case~$(1,1)$. Opposite unlabeled sides are identified
    (cf.~\cite[Figure~1]{Hubert-Schmithusen}).}
    \label{fig:Sab}
\end{figure}
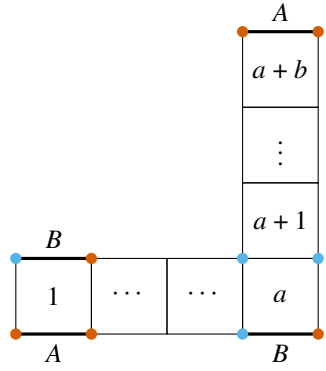

In particular, \Cref{bigthm:asymptotic,bigthm:large_genus} give quadratic
asymptotics for the counting problem on an important family of examples:
Hubert--Weitze-Schmithüsen  staircases~\cite{Hubert-Schmithusen}. These are a family of
$\ZZ$-periodic square-tiled surfaces defined as follows. Let $a \geq 2$, $b \geq
0$ be two integers and define the (compact) square-tiled surface $S_{a,b}$ given
by the permutations
\[
	\sigma_h = (1\, 2\, \cdots\, a)
	\qquad\text{and}\qquad
	\sigma_v = (1\, a\, a+1\, \cdots\, a+b),
\]
see \Cref{fig:Sab,fig:case00,fig:cases2}. 

The genus of $S_{a,b}$ is $2$ unless $(a,b) = (2,0)$, in which case the genus is
one. More precisely, there are three cases, as follows.
\begin{description}
    \item[Case $(0,0)$] If $a = 2$ and $b = 0$, then $S_{a,b} \in \cH(0,0)$,
    as in \Cref{fig:case00}.
    \item[Case $(1,1)$] If $a \geq 3$ and $b \geq 1$, then $S_{a,b} \in
    \cH(1,1)$, as in \Cref{fig:Sab}.
    \item[Case $(2)$] If $a = 2$ and $b \geq 1$, as in \Cref{fig:case2_1}, or $a
    \geq 3$ and $b = 0$, then $S_{a,b} \in \cH(2)$, as in \Cref{fig:case2_2}.
\end{description}

\begin{figure}[tbh]
    \centering
    \begin{subfigure}[b]{.48\textwidth}
    \centering
    \begin{tikzpicture}[scale=.85]
        \foreach \i in {2,...,3}
            \draw (\i,0) rectangle ++(1,1);
        \foreach \i in {1,...,3}
            \draw (3,\i) rectangle ++(1,1);

        \node at (2.5,0.5) {$1$};
        \node at (3.5,0.5) {$2$};
        \node at (3.5,1.5) {$3$};
        \node at (3.5,2.5) {$\vdots$};
        \node at (3.5,3.5) {$2+b$};
        
        \draw[very thick] (2,0) -- ++(1,0) node[midway,below] {$A$};
        \draw[very thick] (3,4) -- ++(1,0) node[midway,above] {$A$};
        \draw[very thick] (3,1) -- ++(-1,0) node[midway,above] {$B$};
        \draw[very thick] (4,0) -- ++(-1,0) node[midway,below] {$B$};
	
        \node[v-bull] at (2,0) {};
        \node[v-bull] at (4,0) {};
        \node[v-bull] at (3,1) {};
        \node[v-bull] at (4,1) {};
        \node[v-bull] at (2,1) {};
        \node[v-bull] at (3,0) {};
        \node[v-bull] at (4,4) {};
        \node[v-bull] at (3,4) {};
    \end{tikzpicture}
    \caption{$a=2$, $b\ge 1$.}
    \label{fig:case2_1}
    \end{subfigure}%
    \hfill %
    \begin{subfigure}[b]{.48\textwidth}
    \centering
    \begin{tikzpicture}[scale=.85]
        \foreach \i in {0,...,3}
            \draw (\i,0) rectangle ++(1,1);

        \node at (0.5,0.5) {$1$};
        \node at (1.5,0.5) {$\cdots$};
        \node at (2.5,0.5) {$\cdots$};
        \node at (3.5,0.5) {$a$};
	
        \draw[very thick] (0,0) -- ++(1,0) node[midway,below] {$A$};
        \draw[very thick] (3,1) -- ++(1,0) node[midway,above] {$A$};
        \draw[very thick] (1,1) -- ++(-1,0) node[midway,above] {$B$};
        \draw[very thick] (4,0) -- ++(-1,0) node[midway,below] {$B$};
        
        \node[v-bull] at (0,1) {};
        \node[v-bull] at (4,1) {};
        \node[v-bull] at (1,0) {};
        \node[v-bull] at (1,1) {};
        \node[v-bull] at (4,0) {};
        \node[v-bull] at (0,0) {};
        \node[v-bull] at (3,1) {};
        \node[v-bull] at (3,0) {};
    \end{tikzpicture}
    \caption{$a\ge 3$, $b=0$.}
    \label{fig:case2_2}
    \end{subfigure}
\caption{Case $(2)$, when $S_{a,b}\in\cH(2)$.}
\label{fig:cases2}
\end{figure}
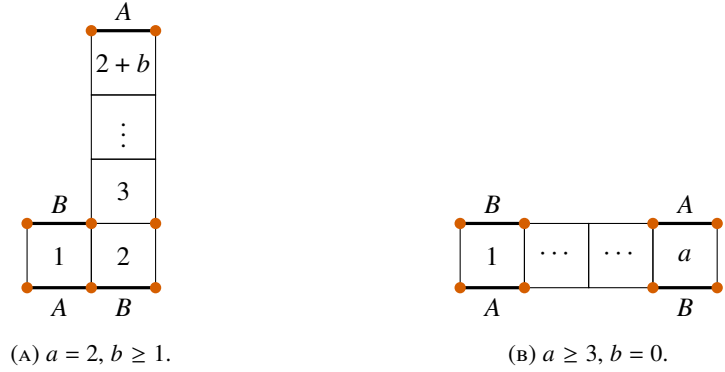

The infinite staircase $\tS = \tS_{a,b}$ is a ramified $\ZZ$-cover over $S =
S_{a,b}$, which is ramified over the set $\Sigma$ of singularities (or marked
points in the Case~$(0,0)$) of $S$.

\subsection{Siegel--Veech constants}\label{sec:SV}
Let $S\in\cH_1(\underline{k})$. Denote by $\cC$ the connected component of
$\cH_1(\underline{k})$ to which $S$ belongs to. Let $\sC(S)$ be the collection
of holonomy vectors that correspond to cylinders on $S$, that is
\[
    \sC(S) = \Set{\hol(\gamma_C)\given C \text{ is a cylinder in } S}.
\]
Then, we are interested in the asymptotic of $N(S,L) = \# (B_L(0) \cap \sC(S))$,
where $B_L(0)$ is the ball of radius $L$ inside of $\RR^2 \cong \CC$.
    
Given a continuous and compactly supported function $f\in C_c(\RR^2)$ we define
the \emph{Siegel--Veech transform} $\hat{f}\colon\cC\to\RR$ of $f$ as 
\[
    \hat{f}(S) = \sum_{\gamma\in\sC(S)} f(\gamma).
\]
Veech showed in~\cite{Veech98}, since $f\in C_c(\RR^2)$, then $\hat{f}$ is in
$L^1(\mu)$ with respect to the Masur--Veech measure $\mu$. 

Then, the \emph{Siegel--Veech formula} proved by Veech in~\cite{Veech98} says
that
\[
    \frac{1}{\vol(\cC)}\int_\cC \hat{f}(S) \, d\mu(S) = c(\cC) \int_{\RR^2} f,
\]
where $c(\cC)$ is called the Siegel--Veech constant (of the counting problem) and
only depends on the connected component (and on the measure $\mu$). Moreover,
Eskin and Masur showed~\cite{Eskin-Masur} that
\[
    \lim_{L\to\infty} \frac{N(S,L)}{\pi L^2} = c(\cC).
\]

\begin{remark}
    Traditionally, Siegel--Veech constants are defined for area one surfaces.
    However, in the case of square-tiled surfaces is often more convenient to normalize
    the area of each square to be one, so that the area of the surface is
    exactly the number of square in it. This implies that one needs to rescale
    the surface and the Siegel--Veech constant accordingly. This is especially relevant in \Cref{sec:q-covers}.
\end{remark}

We conclude this section by giving a formula that allows to compute the
Siegel--Veech constants related to the counting of cylinders in a one-cylinder
composition in a square-tiled surface. This is relevant to our purposes, as we
will show later on that bad cylinders are negligible in the counting and that
good cylinders correspond precisely to one-cylinders.

The following is an adaptation of~\cite[Lemma~2]{Lelievre}, which is based
on~\cite[Proposition~3.10]{Veech89}.

\begin{lemma}\label{lemma:SV_1cyl_cusps}
    The Siegel--Veech constant associated with the counting problem of cylinders
    in a one-cylinder decomposition of a \emph{unit area} primitive square-tiled
    surface $S$ is
    \[
        \frac{1}{2\zeta(2)}\frac{\# \text{ one-cyl cusps of } \slz\cdot S}{\#\,\slz\cdot S}.
    \]
\end{lemma}

\begin{proof}
By~\cite[Lemma~2]{Lelievre}, the Siegel--Veech constant is given by
\begin{equation}\label{eq:Siegel--Veech_Lelievre}
    \frac{1}{2\zeta(2)} \frac{n}{\#\,\slz\cdot S} \sum_{\substack{C \text{ one-cyl cusps} \\ \text{of } \slz\cdot S}} \frac{\cw(C)}{w^2},
\end{equation}
where $n$ is the number of squares tiling the surface $S$. Now, we notice that
the cusp-width of any one-cylinder cusp of a primitive square-tiled surface
tiled by $n$ squares is $n$, which also coincides with the width of the
cylinder. Hence,
\[
	n \sum_{\substack{C \text{ one-cyl cusps} \\ \text{of } \slz\cdot S}} \frac{\cw(C)}{w^2}
	= \# \text{ one-cyl cusps of } \slz\cdot S.\qedhere
\]
\end{proof}

\section{\texorpdfstring{$\ZZ$}{Z}-covers of square-tiled surfaces in
\texorpdfstring{$\cH(2)$}{H(2)}}\label{sec:H2}

In this section we consider the case when the base surface $S$ is a square-tiled
surface inside $\cH(2)$. In particular, this is the case for the
Hubert--Weitze-Schmithüsen  staircases $S_{a,b}$ if $a=2$ and $b\ge 1$ or $a\ge 3$ and
$b=0$, see \Cref{sec:HSstaircases}.

\subsection{Reduction to good cylinders}
Let $\tS$ be $\ZZ$-cover of the surface $S$. We denote by $N(\tS,L)$, the number
of (maximal cylinders of) closed geodesics (up to $\ZZ$-translations) of length
at most $L$, by $N_{\text{good}}(S,L)$ the number of good cylinders in $S$ of
length at most $L$, and $N_{\text{bad}}(S,L)$, of bad cylinders in $S$ of length
at most $L$.
Note that, by the discussion in \Cref{sec:Zcovers}
\[
	N(\tS,L) = N_{\text{good}}(S,L) + N_{\text{bad}}(S,L).
\]
In particular, it is enough to understand the asymptotic behavior of
$N_{\text{good}}(S,L)$ and $N_{\text{bad}}(S,L)$ separately.

The second named author~\cite{Pardo:counting} used this to reduce the counting
problem on wind-tree models to the counting of good cylinders. In
particular,~\cite[Theorem~4.1]{Pardo:counting} implies that, if $S\in\cH(2)$,
then we have that $N_{\text{bad}}(S,L)$ have sub-quadratic growth rate.

\begin{theorem}\label{thm:bad_cyls_subquad}
    Let $S\in\cH(2)$ be a square-tiled surface and $\tS$ a $\ZZ$-cover of $S$.
    Then, the number $N_{\text{bad}}(S,L)$ of bad cylinders of length at most
    $L$ in $S$ have sub-quadratic growth rate, that is, $N_{\text{bad}}(S,L) =
    o(L^2)$.
\end{theorem}

Thus, the counting problem on the infinite staircase is reduced to the counting
of good cylinders, which have quadratic asymptotic growth rate, thanks to
Veech~\cite{Veech89,Veech92} or Eskin and Masur~\cite{Eskin-Masur}. More
precisely:

\begin{corollary}
    Let $S\in\cH(2)$ be a square-tiled surface and $\tS$ a $\ZZ$-cover of $S$.
    Then, the number $N(\tS,L)$ of maximal cylinders of closed geodesics length
    at most $L$ in $\tS$, up to $\ZZ$-translations, have quadratic growth rate,
    \[
        N(\tS,L) = \frac{c_{\text{good}}(S^{(1)})}{\area(S)} \pi L^2,
    \]
    where $c_{\text{good}}(S^{(1)})$ is the Siegel--Veech constant associated
    with the counting problem of good cylinders in $S^{(1)}$, the renormalized
    area~$1$ version of $S$.
\end{corollary}

\subsection{One-cylinder cusps}
Lelièvre~\cite{Lelievre} computed the Siegel--Veech constant for \emph{regular}
cylinders in $\cH(2)$. We recall that a cylinder is regular if both its boundary
components consist of a single saddle connection. Unfortunately, in our setting,
good cylinders are never regular. However, Lelièvre's computations can be
adapted. In particular, it is possible to show that $c_{\text{good}}(S^{(1)})
\to 0$ as $\area(S) \to \infty$ (see~\cite[Remark after Theorem~1]{Lelievre}).

Each square-tiled surface in the stratum $\cH(2)$ decomposes into either one or
two horizontal cylinders. Hence, we can parametrize square-tiled surfaecs in
$\cH(2)$, by the heights, widths, and twist parameters of these cylinders (see,
e.g.,~\cite{Hubert-Lelievre}). Thanks to \Cref{lemma:good_are_one-cyl}, we are
interested in one-cylinder decompositions, where the unique cylinder has three
saddle connections on each boundary component.

Since the surface $S$ is primitive, a cylinder in a one-cylinder decomposition
has height~$1$ and, therefore, width~$n=\area(S)$, the number of squares. It
follows that, in our context, one-cylinder decompositions are parametrized by
the lengths of the three horizontal saddle connections, that add up to $n$, and
a twist parameter, which we set such that the horizontal saddle connections
appear in some (cyclic) order on the bottom of the cylinder, and in reverse
order on the top, see \Cref{f:1-cylinder}.

Later on, we will count the number of primitive one-cylinder square-tiled
surfaces. Hence, let us discuss a bit more these coordinates,
see~\cite{Eskin-Masur-Schmoll, Lelievre-Royer}. Let $\ell_1$, $\ell_2$ and
$\ell_3$ be the lengths of the three saddle connections in the boundary of the
cylinder. Then, if the surface has $n$ squares, these lengths must sum up to a
divisor $r$ of $n$. Now, if the lengths are not equal, there are $r$ possible
values of the twist which yield different surfaces. However, the three possible
cyclic permutations of $(\ell_1, \ell_2, \ell_3)$ give the same surface. If the
three lengths are equal (which is not possible in the primitive case), there are
no cyclic permutations, but only $r/3$ values of the twist result in different
surfaces. Summing up, the number of one-cylinder square-tiled surfaces in
$\cH(2)$ with $n$ squares is
\begin{equation}\label{eq:counting_one_cylinders}
    \frac{1}{3} \sum_{r\mid n} \sum_{\substack{(\ell_1, \ell_2, \ell_3)\in\NN_+^3 \\ \ell_1 + \ell_2 + \ell_3 = r}} r.
\end{equation}

\begin{figure}[tbh]
\centering
\begin{tikzpicture}
	\draw (0,0) -- (1,0) -- (1,.5) -- cycle;
	\draw (1,.5) -- (2,1) --  (2,0) -- (1,0);
        \foreach \i in {2,...,6}
            \draw (\i,0) rectangle ++(1,1);
	\draw (8,.5) -- (7,0) --  (7,1) -- (8,1);
	\draw (8,1) -- (9,1) -- (8,.5) -- cycle;
        
	\draw[->,dashed] (.7,1.3) -- (0,1.3);
	\draw (1,1.3) node {$t$};
	\draw[->,dashed] (1.3,1.3) -- (2,1.3);
	\draw[thick] (2,1) -- node[midway,above] {$\ell_3$} ++(4,0) -- node[midway,above] {$\ell_2$} ++(2,0) -- node[midway,above] {$\ell_1$} ++(1,0);
	\draw[thick] (0,0) -- node[midway,below] {$\ell_1$} ++(1,0) -- node[midway,below] {$\ell_2$} ++(2,0) -- node[midway,below] {$\ell_3$} ++(4,0);
    
    \node[b-bull] at (0,0) {};
    \node[b-bull] at (1,0) {};
    \node[b-bull] at (3,0) {};
    \node[b-bull] at (7,0) {};
    \node[b-bull] at (2,1) {};
    \node[b-bull] at (6,1) {};
    \node[b-bull] at (8,1) {};
    \node[b-bull] at (9,1) {};
\end{tikzpicture}
\caption{A one-cylinder decomposition of a surface in $\cH(2)$ parametrized by the lengths of its boundary saddle connections and a twist parameter (cf.~\cite[Figure~2.3]{Lelievre-Royer}).}
\label{f:1-cylinder}
\end{figure}
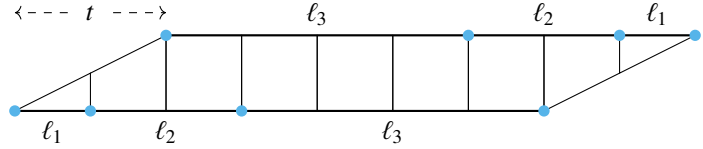

\subsection{Computing the Siegel--Veech constants in
\texorpdfstring{$\cH(2)$}{H(2)}}

\Cref{lemma:SV_1cyl_cusps} gives us a formula to compute the Siegel--Veech
constants when $S\in\cH(2)$. In our context, we need to understand the
quantities
\[
	o(S) = \#\slz\cdot S
	\qquad\text{and}\qquad
	c_1(S) = \# \text{ one-cyl cusps of } \slz\cdot S
\]
for the primitive square-tiled surfaces $S \in \cH(2)$.

$\slz$-orbits of primitive square-tiled surfaces $S \in \cH(2)$ are determined
by the number $n \in \NN$ of squares used to tile $S$ and, when $n$ is odd, by
the number of integer Weierstrass points. See, e.g.,~\cite{Hubert-Lelievre}. For
even $n$, let $C_n$ denote the only $\slz$-orbit of primitive square-tiled
surfaces in $\cH(2)$ tiled by $n$ squares. For odd $n$, let $A_n$ denote the
$\slz$-orbit of primitive square-tiled surfaces in $\cH(2)$ tiled by $n$
squares, with exactly one integer Weierstrass point (which corresponds to the
conical singularity), and, by $B_n$, the other $\slz$-orbit of primitive
square-tiled surfaces in $\cH(2)$ tiled by $n$ squares, which have three integer
Weierstrass points.

\begin{proposition}\label{prop:one_cylinder_Siegel--Veech}
    The Siegel--Veech constant associated with the counting problem of cylinders
    in a one-cylinder decomposition of a \emph{unit area} primitive square-tiled
    surface $S$ is
    \[
        \begin{cases}
        \displaystyle\frac{1}{9 \zeta(2)}\cdot\frac{1}{n - 1}, & \text{if } S \in A_n, \\[1em]
        \displaystyle\frac{1}{3 \zeta(2)}\cdot\frac{1}{n - 3}\left(1 - \frac{4}{\psi(n)}\right), & \text{if } S \in B_n,\\[1em]
        \displaystyle\frac{2}{9 \zeta(2)}\cdot\frac{1}{n - 2}\left(1 - \frac{3}{\psi(n)}\right), & \text{if } S \in C_n,
        \end{cases}
    \]
    where $\psi$ is the Dedekind psi function
    \[
        \psi(n) = n\prod_{p \mid n}\left(1 + \frac{1}{p}\right),
    \]
    where the notation $\prod_{p \mid n}$ means that we are taking the product
    over the prime factors of $n$.
\end{proposition}

Before we begin, let us recall that as an application of Möbius inversion, by~\cite[Lemma~3.9]{Lelievre-Royer}, the counting of primitive one-cylinder surfaces with $n$ squares can be reduced to counting \emph{height-primitive} ones, through the identity
\begin{multline}\label{eq:prim-id}
    \# \,\text{one-cylinder primitive square-tiled surfaces with $n$ squares}
    \\
    =
    \sum_{r\mid n} \mu\left(\frac{n}{r}\right) \# \,\text{one-cylinder square-tiled surfaces with $r$ squares and height $1$},
\end{multline}
where $\mu$ is the Möbius function, and both counting are over horizontal cylinder decompositions of square-tiled surfaces of the same \emph{type} ($A$, $B$, or $C$).

Also note that, for $k \geq 0$, we have
\begin{equation}\label{eq:Mob-id}
    \sum_{d \mid n} \frac{\mu(r)}{d^k} = \prod_{p \mid n} \left( 1 - \frac{1}{p^k}\right).
\end{equation}

\begin{proof}
If $S \in \cH(2)$ is a square-tiled surface, tiled by $n\in\NN$ squares, then,
by~\cite[Remark after Lemma~4.11]{Eskin-Masur-Schmoll}
and~\cite[Theorem~1.1]{Lelievre-Royer},
\begin{equation}\label{eq:1_cyl_orbit_count}
	o(S) =
	\begin{cases}
	\displaystyle\frac{3}{16}(n - 1)n^2 \prod_{p \mid n} \left(1 - \frac{1}{p^2} \right), & \text{if } S \in A_n, \\
	\displaystyle\frac{3}{16}(n - 3)n^2 \prod_{p \mid n} \left(1 - \frac{1}{p^2} \right), & \text{if } S \in B_n,\\
	\hfill\displaystyle\frac{3}{8}(n - 2)n^2 \prod_{p \mid n} \left(1 - \frac{1}{p^2} \right), & \text{if } S \in C_n.
	\end{cases}
\end{equation}
For $S \in A_n$, following~\cite[Lemma~5.7]{Lelievre-Royer}, gives
\[
    c_1(S) = \alpha(n) = \frac{1}{3} \#\Set{\ell_1, \ell_2, \ell_3\in 2\NN+1\given \gcd(\ell_1,\ell_2,\ell_3) = 1,\, \ell_1 + \ell_2 + \ell_3 = n}.
\]
Writing each $\ell_i = 2k_i + 1$ for integers $k_i \geq 0$, the sum $\ell_1 + \ell_2 + \ell_3 = n$ becomes
\[
    k_1 + k_2 + k_3 = \frac{n-3}{2},
\]
and the number of non-negative integer solutions to this is
\[
    \binom{\frac{n-3}{2} + 3 - 1}{3 - 1}
    = \binom{\frac{n+1}{2}}{2}
    = \frac{1}{2}\left(\frac{n+1}{2}\right)\left(\frac{n-1}{2}\right)
    = \frac{n^2 - 1}{8}.
\]
Now, every unconstrained triple summing to $n$ has a greatest common divisor $d = \gcd(\ell_1, \ell_2, \ell_3) \mid n$.

Dividing the elements of an unconstrained triple by its greatest common divisor $d$ yields then a bijection to a triple of odd integers summing to $n/d$ with a greatest common divisor equals to $1$. It follows that
\[
    \sum_{d \mid n} 3 \alpha\left(\frac{n}{d}\right) = \frac{n^2 - 1}{8}
\]
Applying Möbius inversion formula, we obtain
\[
    \alpha(n) = \sum_{d \mid n} \mu\left(\frac{n}{d}\right) \frac{d^2 - 1}{24} = \frac{1}{24} \sum_{d \mid n} \left(\mu(d) \left(\frac{n}{d}\right)^2 - \sum_{d \mid n} \mu(d) \right)
\]
Finally, by \cref{eq:Mob-id}, we conclude that
\[
    c_1(S) = \frac{n^2}{24} \prod_{p \mid n} \left(1 - \frac{1}{p^2}\right).
\]
We remark that in~\cite{Lelievre-Royer}, there is a factor $n^3$ instead of a
$n^2$. The difference stems from the fact that we are counting the number of
$1$-cylinder \emph{cusps}, whereas Lelièvre--Royer are counting \emph{surfaces}. Since, as we
noted in \Cref{lemma:SV_1cyl_cusps}, for every one-cylinder cusp there are
exactly $n$ primitive square-tiled surfaces tiled by $n$ squares, we obtain the
above formula.

We now consider the case when $S \in B_n$. In this case, we have to choose three
integer lengths for the saddle connections in the boundary of the cylinders. By
definition of $B_n$, two of these lengths are even, and one of them is odd.
Then,
\[
    c_1(S) = \#\Set{\ell_1\in 2\NN+1, \ell_2, \ell_3\in 2\NN+2 \given \gcd(\ell_1,\ell_2,\ell_3) = 1,\, \ell_1 + \ell_2 + \ell_3 = n},
\]
and, similarly to the previous case, we have
\[
\begin{split}
    c_1(S) &= \sum_{r\mid n} \mu\left(\frac{n}{r}\right) \#\Set{\ell_1\in 2\NN+1, \ell_2, \ell_3\in 2\NN+2 \given \ell_1 + \ell_2 + \ell_3 = r} \\
        &= \sum_{r\mid n} \mu\left(\frac{n}{r}\right) \#\Set{(\ell_1, \ell_2, \ell_3)\in \NN^3\given \ell_1 + \ell_2 + \ell_3 = \frac{r-5}{2}} \\
        &= \sum_{r\mid n} \mu\left(\frac{n}{r}\right) \frac{r^2-4r+3}{8}\\
        &= \frac{1}{8}\sum_{d\mid n} \left(\mu(d)\left(\frac{n}{d}\right)^2 - 4\mu(d)\frac{n}{d} + 3 \mu(d)\right)\\
        &= \frac{n^2}{8}\prod_{p \mid n} \left(1 - \frac{1}{p^2}\right) - \frac{n}{2} \prod_{p \mid n} \left(1 - \frac{1}{p}\right).
\end{split}
\]
We remark that the factor $1/3$ from \cref{eq:counting_one_cylinders} is missing in the previous equation as we have specified the cyclic permutation $(\ell_1,\ell_2,\ell_3)$ to begin with.

The final case is when $S\in C_n$. By definition, $n$ is even. Moreover, in the primitive case, we cannot have that all the three lengths of the saddle connections
are even. Hence, two of them are odd and one of them is even:
\[
    c_1(S) = \gamma(n) = \#\Set{\ell_1,\ell_2\in 2\NN+1, \ell_3\in 2\NN+2 \given \gcd(\ell_1,\ell_2,\ell_3) = 1,\, \ell_1 + \ell_2 + \ell_3 = n}.
\]
In particular, the greatest common divisor of the lengths is always odd and, conversely, no triple divisible by an even number can represent a primitive square-tiled surface in $C_n$.
That is,
\[
    \gamma(n) = \sum_{\substack{d \mid n \\ d \text{ odd}}} \gamma\left(\frac{n}{d}\right)
\]

Therefore, in the Möbius inversion only the odd divisors $d$ of $n$ appear, and for them $r = n/d$ is even. Then, similarly to the above, we have
\[
\begin{split}
    c_1(S) &= \sum_{\substack{r\mid n \\ n/r \text{ odd}}} \mu\left(\frac{n}{r}\right) \#\Set{\ell_1,\ell_2\in 2\NN+1, \ell_3\in 2\NN+2 \given \ell_1 + \ell_2 + \ell_3 = r}\\
        &= \sum_{\substack{r\mid n \\ n/r \text{ odd}}} \mu\left(\frac{n}{r}\right) \#\Set{\ell_1,\ell_2\in 2\NN+1, \ell_3\in 2\NN+2 \given \ell_1 + \ell_2 + \ell_3 = \frac{r - 4}{2}}\\
        &= \sum_{\substack{r\mid n \\ n/r \text{ odd}}} \mu\left(\frac{n}{r}\right) \frac{r^2-2r}{8}\\
        &= \frac{1}{8}\sum_{\substack{d\mid n \\ d \text{ odd}}} \left(\mu(d) \left(\frac{n}{d}\right)^2 - 2 \mu(d)\frac{n}{d}\right)\\
        &= \frac{n^2}{8}\prod_{\substack{p \mid n \\ p \neq 2}} \left(1 - \frac{1}{p^2}\right) - \frac{n}{4} \prod_{\substack{p \mid n \\ p \neq 2}} \left(1 - \frac{1}{p}\right)\\
        &= \frac{n^2}{6}\prod_{p \mid n} \left(1 - \frac{1}{p^2}\right) - \frac{n}{2} \prod_{p \mid n} \left(1 - \frac{1}{p}\right).
\end{split}
\]

Summing up, we have shown that
\begin{equation}\label{eq:1_cyl_cusps_count}
    c_1(S) =
	\begin{cases}
	\displaystyle\frac{n^2}{24} \prod_{p \mid n} \left(1 - \frac{1}{p^2} \right), & \text{if } S \in A_n,\\
	\displaystyle\frac{n^2}{8} \prod_{p \mid n} \left(1 - \frac{1}{p^2}\right) - \frac{n}{2} \prod_{p \mid n} \left(1 - \frac{1}{p}\right), & \text{if } S \in B_n,\\
	\displaystyle\frac{n^2}{6} \prod_{p \mid n} \left(1 - \frac{1}{p^2}\right) - \frac{n}{2} \prod_{p \mid n} \left(1 - \frac{1}{p}\right),  & \text{if } S \in C_n.
	\end{cases}
\end{equation}

The conclusion now follows from \Cref{lemma:SV_1cyl_cusps}, using \cref{eq:1_cyl_cusps_count,eq:1_cyl_orbit_count}.
\end{proof}

Using the results of this section we obtain our first main result,
\Cref{bigthm:asymptotic} when the base surface is in $\cH(2)$.

\begin{theorem}
    Let $S$ be a square-tiled surface in $\cH(2)$ and consider a zero holonomy
    $\ZZ$-cover $\tS$. Then, the number $N(\tS,L)$ of maximal cylinders of
    closed geodesics of length at most $L$ in $\tS$, up to $\ZZ$-translations
    has quadratic growth:
    \[
        N(\tS,L) \sim \frac{c_\text{good}(S^{(1)})}{\area(S)} \cdot\pi L^2,
    \]
    where $c_\text{good}(S^{(1)})$ is the Siegel--Veech constant which counts
    one-cylinder directions cylinders in the area $1$ normalization of $S$, and
    can be explicitly computed.
\end{theorem}

The above theorem implies the quadratic counting for Hubert--Weitze-Schmithüsen 
staircases, with the explicit Siegel--Veech constants:

\begin{corollary}\label{cor:SV_staircases_H2}
    Let $S = S_{a,b}$, for some $a,b \in \NN$ be one of the Hubert--Weitze-Schmithüsen 
    staircases. Assume $S_{a,b}$ is in $\cH(2)$, meaning that $a = 2$ and $b
    \geq 1$, or $a \geq 3$ and $b = 0$. Then, the number $N(S,L)$ of maximal
    cylinders of closed geodesics of length at most $L$ in $S$, up to $\ZZ$
    translations, has quadratic growth rate,
    \[
        N(S,L) \sim 
        \pi L^2 \cdot \begin{cases}
        \displaystyle\frac{1}{9 \zeta(2)}\cdot\frac{1}{(a+b)(a+b - 1)}, & \text{if } a + b \text{ is odd}, \\[1em]
        \displaystyle\hfil\frac{2}{9 \zeta(2)}\cdot\frac{1}{(a+b)(a+b - 2)}\left(1 - \frac{3}{\psi(a+b)}\right), & \text{if } a + b \text{ is even}.
        \end{cases}
    \]
\end{corollary}

\begin{proof}
We first consider the case when the number of squares $n = a + b$ is odd. If
$a=2$, then $b$ is odd. In this case, the Weierstrass point inside the cylider
of height $b$ is non integral. If $b=0$ and $a\ge 3$, the only horizontal
cylinder is bounded by $3$ saddle connections, two of them have length $1$ and
the middle one has to be of odd length. Hence, in both cases, there is only
integral Weierstrass point (the singularity), that is, the surface is in the
orbit $A_n$.

If $n$ is even, the by~\cite{Hubert-Lelievre}, the surface belongs to the orbit
$C_n$. Applying \Cref{prop:one_cylinder_Siegel--Veech}, we are done.
\end{proof}

\section{Finite intermediate cyclic covers with base in
\texorpdfstring{$\cH(2)$}{H(2)}}\label{sec:q-covers}
Let $q$ be an odd prime and consider $\qS$, the cyclic cover of degree $q$ of
$S\in\cH(2)$ defined by $\eta$. For $L>0$, we denote, as usual, by $N(\qS, L)$
the number of cylinders on $\qS$ of length at most $L$. By the work of
Veech~\cite{Veech89}, there exists a Siegel--Veech constant $c(\qS)$ such that
$N(\qS,L) \sim c(\qS) \pi L^2$.

In this section we show \Cref{bigthm:large_genus} in the case when the base is
in $\cH(2)$. Namely, that, as $q$ goes to infinity, the sequence of
Siegel--Veech constants $c(\qS)$ converges to the Siegel--Veech constant of the
infinite cover $c(\tS)$.

Let $\cM$ be the $\slr$-orbit of $S$ and $\cM^q$, the $\slr$-orbit of $\qS$.
Then, elements in $\cM^q$ consist of cyclic covers of elements in $\cM$. If we
denote by $K$ the zero holonomy part of the homology of $S$, $H_1^{(0)}(S;\ZZ)$.
Then the Veech group of $S$, $\Gamma(S)$, acts as a finite index subgroup
$\Gamma'$ of $\slz$ on $K \cong \ZZ^2$.

By~\cite[Theorem~0]{Matthews-Vaserstein-Weisfeiler}, any finitely generated
Zariski dense subgroup $\Gamma' < \slz$ projects subjectively onto
$\SL_2(\ZZ_q)$ for any $q \in \NN$ coprime to some fixed $q_0 \in \NN$
(depending on $\Gamma'$). However, by~\cite[Theorem~2.8]{Kirkwood-McDonald}, for
every odd $q \in \NN$, $\SL_2(\ZZ_q)$ acts transitively on unimodular elements
of $\ZZ_q^2$. Then, it follows that $\cM^q$ contains \emph{all} $\ZZ_q$-covers of
$S$. Or, what is the same, $\cM^q$ consists of \emph{every} $\ZZ_q$-cover of
elements in $\cM$ (cf.~\cite[footnote~7]{Magee}).

In the cyclic covers case, not only we need to distinguish good and bad
cylinders (modulo $q$), but we also have to deal with \enquote{strips modulo
$q$}. Let us give the formal definitions. Given a cylinder $C$ in $S$ with core
curve $\gamma_C$, we say that
\begin{enumerate}[(a)]
    \item $C$ is \emph{good modulo $q$} if $\pairing{\gamma_C, k} \modn{q} 0$
    for all $k\in K$. 
    \item $C$ is \emph{bad modulo $q$} if $C$ is not good modulo $q$, but $\pairing{\gamma_C, \eta} \modn{q} 0$.
    \item $C$ is a \emph{strip modulo $q$} if $\pairing{\gamma_C, \eta} \nmodn{q} 0$.
\end{enumerate}

We remark that both good, and bad $\bmod\;q$ cylinders lift as $q$ isometric
copies to $\qS$, while strips $\bmod\;q$ lift as one cylinder $q$-times longer.

\begin{lemma}\label{lemma:good-q}
For every prime $q$ large enough, a cylinder is good modulo $q$ if and only if it is good.
\end{lemma}

\begin{proof}
$\Gamma(S)$ acts on the set of cylinders of $S$ with finitely many orbits (there are finitely many cusps, and finitely many cylinders in each periodic direction), and $\proj_K$ is $\Gamma(S)$-equivariant.
The set $A$ of all cylinders in $S$ that are not good is $\Gamma(S)$-invariant.
Moreover, $C \in A$ if and only if $\proj_K[\gamma_C] \neq 0$.

Fix a cylinder $C$ in $S$ and let $d(C) \geq1$ be the greatest common divisor of the coordinates of $\proj_K[\gamma_C]$ in any basis of $K = H_1^{(0)}(S;\ZZ)$.
Since $\Gamma(S)$ acts on $K$ by linear automorphisms and $d(C)$ is invariant under multiplication by invertible integral matrices, $d(C)$ only depends on the $\Gamma(S)$-orbit of $C$.
It follows that $d$ takes finitely many values. Let $d_0 = \max\Set{d(C) \given C \in A} < \infty$.
Hence, for every prime $q > d_0$, we have $\proj_K[\gamma_C] \nmodn{q} 0$ for every $C\in A$.
That is, for $q > d_0$, $\proj_K[\gamma_C] \neq 0$ if and only if $\proj_K[\gamma_C] \nmodn{q} 0$, or, what is the same, a cylinder is good modulo $q$ if and only if it is good.
\end{proof}

We are now ready to begin the proof of \Cref{bigthm:large_genus}, assuming that
the base surface is in $\cH(2)$:

\begin{theorem}\label{thm:large_genus_H2}
    Let $S\in\cH(2)$ be a square-tiled surface and let $\eta\in
    H_1(S,\Sigma;\ZZ)$ be a zero holonomy curve. Let $q$ be a prime number and
    $\qS$ the $\ZZ_q$-cover of $S$ determined by $\eta$. Then
    \[
        \frac{1}{q} c(\qS) \xrightarrow{q\to\infty} c(\tS),
    \]
    where $c(\qS)$ is the Siegel--Veech constant associated to $N(\qS,L)$ and
    $\tS$ is the $\ZZ$-cover of $S$ determined by $\eta$.
\end{theorem}

\begin{proof}
Let $A$ be the set of all cylinders in $S$ that are not good ($A$ is the set of
cylinders in a two-cylinder decomposition), $B_q$ the set of all bad cylinders
$\bmod\;q$, and $C_q(k)$, the set of all cylinders $C$ in $S$ such that
$\proj_K[\gamma_C] \modn{q} k$, where $k \in K$. By \Cref{lemma:good-q}, for $q$
large enough, it follows that
\[
 	A
	\quad =\quad
    \bigcup_{\mathclap{[k]_q\in K^\times_S(\ZZ_q)}} \quad C_q(k),
\]
where $K^\times_S(\ZZ_q)$ stands for the set of unimodular elements in
$K/\modn{q} \cong \ZZ_q^2$. Notice that $A$ is the union of $q^2-1$ classes.
Similarly, $B_q$ is the union of $q-1$ classes: since $K$ is symplectic, the
condition $ \langle h, \eta \rangle \modn{q} 0$ is of codimension~$1$ in
$K(\ZZ_q)$. 

By the discussion above \Cref{lemma:good-q}, for $q$ large enough, $\Gamma(S)$ acts transitively on $K^\times_S(\ZZ_q)$, all $\ZZ_q$-covers (defined by cocycles in $K$) are in the same orbit, and the Siegel--Veech constants for
$C_q(k)$ do not depend on $k \in K$. Thus, we call it $c(C_q)$. Then, we have
that
\begin{equation}\label{eq:SV_notgood}    
	c(A) = (q^2-1) c(C_q)
    \qquad\text{and}\qquad
	c(B_q) = (q-1) c(C_q).
\end{equation}
Also, let $G$ be the set of good cylinders in $S$. Recall that the Siegel--Veech
constant $c(G)$ (resp.\ $c(A)$) corresponds to the counting of cylinders in
one-cylinder decompositions (resp.\ two-cylinder decompositions\footnote{In
principle, one could compute the constant $c(A)$ as in the one-cylinder case,
and that most of the work is already done in~\cite{Lelievre}. We do not do this,
since we do not need it.}).

We now lift these constants to $\qS$. We consider that the area of $\qS$ is $q$ times the area of $S$. Then, cylinders in $G$ and in $B_q$ lift as $q$~isometric copies to $\qS$, and we obtain
\[
    c(\widetilde G) = q c(G)
    \qquad\text{and}\qquad
    c(\widetilde B_q) = q c(B_q),
\]
where, for a family of cylinders $X$ in $S$, $\widetilde X$ denotes the lift of
$X$ to $\qS$.

Now, let $T_q$ be the set of strips mod~$q$ in $S$. Then $T_q = A \setminus
B_q$, and 
\[
    c(T_q) = c(A) - c(B_q).
\]
As we recalled previously, cylinders in $T_q$ lift as one cylinder $q$ times
longer. It follows that
\[
    q^2 c(\widetilde T_q) = c(T_q).
\]
    
Now, the Siegel--Veech constant for $\qS$ satisfies $c(\qS) = c(\widetilde G) +
c(\widetilde B_q) + c(\widetilde T_q)$. In fact, every cylinder in $\qS$
descends to a cylinder in $S$, and $\{G, B_q, T_q\}$ is a partition of all
cylinders in $S$. By the previous discussion, it follows that
\begin{equation}\label{eq:q-Siegel--Veech}
    c(\qS) = q c(G) + q c(B_q) + \frac{1}{q^2} c(T_q),
\end{equation}
where $c(G)$, $c(B_q)$ and $c(T_q)$ are Siegel--Veech constants associated with
$G$, $B_q$ and $T_q$ in $S$, respectively. Moreover, by \cref{eq:SV_notgood}, we
have, for $q$ prime large enough, that
\begin{equation}\label{eq:q-bad-Siegel--Veech}
    c(B_q) = \frac{q-1}{q^2-1} c(A) = \frac{1}{q+1} c(A)
    \qquad\text{and}\qquad
    c(T_q) = \frac{q}{q+1} c(A).
\end{equation}
Finally,
\[
    c(\qS) = q c(G) + \frac{q^2+1}{q(q+1)} c(A),
\]
and
\[
    \lim_{\substack{q\to\infty \\ q \text{ prime}}} \frac{1}{q} c(\qS) = c(G) = c(\tS),
\]
which was what we wanted to show.
\end{proof}

\section{\texorpdfstring{$\ZZ$}{Z}-covers of square-tiled surfaces in
\texorpdfstring{$\cH(1,1)$}{H(1,1)}}\label{sec:H11}

In this section, we consider the case when $S\in\cH(1,1)$, dealing with relative
homology. We begin by showing the analogue of \Cref{thm:bad_cyls_subquad} for
$\cH(1,1)$, that is, that bad cylinders grow subquadratically.

\begin{theorem}\label{thm:bad_cyls_subquad_H11}
    Let $S$ be a surface in $\cH(1,1)$. Then, the number of bad cylinders with
    core curve of length at most $L$, $N_\text{bad}(S,L)$ grows
    subquadratically: $N_\text{bad}(S,L) = o(L^2)$.
\end{theorem}

\begin{proof}
    By \cref{eq:homology_splitting} on \cpageref{eq:homology_splitting}, the
    absolute homology of $S$ splits into 
    \[
        H_1(S;\ZZ) = H_1^{\text{st}}(S;\ZZ) \oplus H_1^{(0)}(S;\ZZ).
    \]
    Hence, once we identify $H_1^{\text{st}}(S;\ZZ)$ with $\ZZ^2$, the action of
    the Veech group $\Gamma(S)$ on is simply the standard action of $\slz$ on
    $\ZZ^2$.

    There is a short exact sequence
    \begin{equation}\label{eq:short_exact_sequence}
        0 \to H_1 (S;\ZZ) \hookrightarrow H_1 (S, \Sigma; \ZZ) \twoheadrightarrow \widetilde{H}_0 (\Sigma;\ZZ) \to 0,
    \end{equation}
    where $\widetilde{H}_0 (\Sigma;\ZZ)$ is the $0^\text{th}$ reduced homology
    of $\Sigma$, meaning the assignation of integral weights to points in
    $\Sigma$ whose sum is $0$. Since $|\Sigma| = 2$, $\widetilde{H}_0
    (\Sigma;\ZZ) \cong \ZZ$.

    Let us denote by $\rho_\Gamma^{(0)}$ the representation of $\Gamma(S)$, the
    Veech group of $S$, obtained by looking at its action on the zero holonomy
    relative homology $H_1^{(0)}(S, \Sigma; \ZZ)$. For more details on this
    action, we refer to~\cite[Section~5]{Hooper-Weiss}. By looking at the
    analogue of~\eqref{eq:short_exact_sequence} on the zero holonomy homology,
    we see that the image of $\rho_\Gamma^{(0)}$ lies inside $\slz\rtimes\ZZ^2$.

    By a theorem of Möller (see~\cite[Theorem~38]{Betal}), since $S$ is of genus
    $2$, the action of $\Gamma(S)$ on $H_1^{(0)}(S;\ZZ)$ is \emph{arithmetic},
    meaning that it forms a finite index subgroup of $\slz$. This implies, that
    the image of $\rho^{(0)}_\Gamma$ is the semidirect product of a finite index
    subgroup of $\slz$ with a lattice inside of $\ZZ^2$.

    We now have two cases: if the lattice is trivial, then we have a $\Gamma(S)$-invariant splitting of the zero holonomy homology\footnote{By a result of Filip~\cite[Theorem~5.4.7]{Filip:survey}, this \emph{split} case can only happen for surfaces lying in a finite number of \enquote{atypical} affine invariant submanifolds of $\cH(1,1)$; these submanifolds may, however, contain infinitely many $\slz$-orbits of square-tiled surfaces.
    In particular, this is the case whenever the surface $S$ is a double cover of a surface in $\cH(0,0)$, ramified over the marked points.}
    \begin{equation}\label{eq:equivariant_splitting}        
        H^{(0)}_1 (S, \Sigma; \ZZ) = H^{(0)}_1 (S; \ZZ) \oplus H^{(0), \text{rel}}_1 (S, \Sigma; \ZZ),
    \end{equation}
    otherwise, we obtain a finite index subgroup of the whole semidirect product
    $\slz\rtimes\ZZ^2$.

    In the split case, if $\eta \in H^{(0), \text{rel}}_1 (S, \Sigma; \ZZ)$, then we have no bad cylinders and we are done.
    Otherwise, we can project the curve
    $\eta$ defining the $\ZZ$-cover non-trivially to the absolute homology $H^{(0)}_1 (S; \ZZ)$ and the result follows from~\cite[Theorem~4.1]{Pardo:counting}.

    We now turn to the latter \emph{non-split} case, in which the image of
    $\rho_\Gamma^{(0)}$ is a finite index subgroup of $\slz\rtimes\ZZ^2$. This
    implies that the Zariski closure of this subgroup is the whole
    $\slr\rtimes\RR^2$. Then, using strong approximation as in the beginning of
    \Cref{sec:q-covers}, for any sufficiently large prime $q$, the action of
    $\Gamma(S)$ on the zero holonomy homology modulo $q$, $H^{(0)}_1 (S, \Sigma;
    \ZZ_q)$ is transitive. This implies, using the terminology of~\cite{ACMSS},
    that every orbit is \emph{unsplit}, following their~Proposition~4.2. With
    this, the proofs of~\cite[Propositions~4.5 and~4.6]{ACMSS} apply almost
    verbatim, replacing their~$H_1^{\text{prim}}(X,\mathcal{B};\ZZ/p^k\ZZ)$ by
    $H^{(0)}_1(S,\Sigma;\ZZ_q)\setminus\{0\}$, with $q$ prime large enough. This
    show that bad cylinders modulo $q$ have codimension $1$ inside the homology,
    and their number is $O(1/q)$. In particular, bad cylinders grow
    subquadratically.
\end{proof}

A byproduct of the previous proof is that \Cref{bigthm:large_genus} holds in the
case when the base is in $\cH(1,1)$ following the proof of
\Cref{thm:large_genus_H2} for the case when the base is in $\cH(2)$.

\begin{corollary}\label{cor:large_genus_H11}
    Let $S\in\cH(1,1)$ be a square-tiled surface and let $\eta\in
    H_1(S,\Sigma;\ZZ)$ be a zero holonomy curve. Let $q$ be a prime number and
    $\qS$ the $\ZZ_q$-cover of $S$ determined by $\eta$. Then
    \[
        \frac{1}{q} c(\qS) \xrightarrow{q\to\infty} c(\tS),
    \]
    where $c(\qS)$ is the Siegel--Veech constant associated to $N(\qS,L)$ and
    $\tS$ is the $\ZZ$-cover of $S$ determined by $\eta$.
\end{corollary}

Combining \Cref{cor:large_genus_H11} with \Cref{thm:large_genus_H2} we have
completed the proof of \Cref{bigthm:large_genus}. The above proves
\Cref{bigthm:asymptotic} for the last open case: when the base surface is in
$\cH(1,1)$.

\begin{theorem}
    Let $S$ be a square-tiled surface in $\cH(1,1)$ and consider a zero holonomy
    $\ZZ$-cover $\tS$. Then, the number $N(\tS,L)$ of maximal cylinders of
    closed geodesics of length at most $L$ in $\tS$, up to $\ZZ$-translations
    has quadratic growth:
    \[
        N(\tS,L) \sim c(\tS) \pi L^2,
    \]
    and the Siegel--Veech constant $c(\tS)$ which counts good cylinders in $S$.
\end{theorem}

Going back to the staircase surfaces, in~\cite{Hubert-Schmithusen} it is shown
that, if $a$ or $b$ are even, then there are $1$-cylinders directions. We show
that there are \emph{no} $1$-cylinder directions in the other cases.

\begin{proposition}\label{prop:no_1cyl}
    Let $a\ge 2$ and $b\ge 0$ two odd integers. Then, there are no $1$-cylinder
    directions in the square-tiled surface $S_{a,b}$. 
\end{proposition}

\begin{figure}[tbh]
    \centering
    \begin{tikzpicture}[mid arrow/.style={
			   postaction={decorate,decoration={
				   markings,
				   mark=at position .5 with {\arrow{latex}}
		   }}}]
        \draw (0,0) rectangle ++(1,1);
        \draw (1,0) rectangle ++(1,1);
        \draw (2,0) rectangle ++(1,1);
        \draw (2,1) rectangle ++(1,1);

        \node[cross out, draw, inner sep=2pt, color=OIVermilion] at (0,0.5) {};
        \node[cross out, draw, inner sep=2pt, color=OIVermilion] at (0.5,1) {};
        \node[cross out, draw, inner sep=2pt, color=OIVermilion] at (1.5,1) {};
        \node[cross out, draw, inner sep=2pt, color=OIVermilion] at (1.5,0.5) {};
        \node[cross out, draw, inner sep=2pt, color=OIVermilion] at (2,1.5) {};
        \node[cross out, draw, inner sep=2pt, color=OIVermilion] at (2.5,1.5) {};
    \end{tikzpicture}
    \caption{The six Weierstrass points on the surface $S_{3,1}$.}
    \label{fig:no_1cyl_a}
\end{figure}
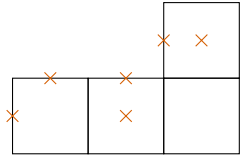

\begin{proof}[Proof of \Cref{prop:no_1cyl}]
    Assume by contradiction that there were a $1$-cylinder direction. Since
    $a+b$ is even, Weierstrass point on the surface would have the following
    pattern, see \Cref{fig:no_1cyl_a}:
    \begin{itemize}
        \item no integral points.
        \item $2$ points on horizontal and vertical sides.
        \item $2$ points on centers of some squares.
    \end{itemize}
    However, up to an affine deformation of the surface, sending the
    $1$-cylinder direction to the horizontal one, we would have one of the two
    configurations in \Cref{fig:no_1cyl_b}, in which there are $4$ points on
    horizontal sides and $2$ either on vertical sides or on the centers of some
    squares. Since the unordered list of Weierstrass points on horizontal and
    vertical sides, and centers of squares is an affine invariant,
    see~\cite{Hubert-Lelievre}, we have a contradiction.
\end{proof}

\begin{figure}[tbh]
    \centering
    \begin{tikzpicture}[mid arrow/.style={
			   postaction={decorate,decoration={
				   markings,
				   mark=at position .5 with {\arrow{latex}}
		   }}}]
        \draw (0,0) rectangle ++(1,1);
        \draw (1,0) rectangle ++(1,1);
        \draw (2,0) rectangle ++(1,1);
        \draw (3,0) rectangle ++(1,1);

        \node[cross out, draw, inner sep=2pt, color=OIVermilion] at (0,0.5) {};
        \node[cross out, draw, inner sep=2pt, color=OIVermilion] at (0.5,1) {};
        \node[cross out, draw, inner sep=2pt, color=OIVermilion] at (1.5,1) {};
        \node[cross out, draw, inner sep=2pt, color=OIVermilion] at (2,0.5) {};
        \node[cross out, draw, inner sep=2pt, color=OIVermilion] at (2.5,1) {};
        \node[cross out, draw, inner sep=2pt, color=OIVermilion] at (3.5,1) {};

        \draw (7,0) rectangle ++(1,1);
        \draw (8,0) rectangle ++(1,1);
        \draw (9,0) rectangle ++(1,1);
        \draw (10,0) rectangle ++(1,1);

        \node[cross out, draw, inner sep=2pt, color=OIVermilion] at (7.5,0.5) {};
        \node[cross out, draw, inner sep=2pt, color=OIVermilion] at (7.5,1) {};
        \node[cross out, draw, inner sep=2pt, color=OIVermilion] at (8.5,1) {};
        \node[cross out, draw, inner sep=2pt, color=OIVermilion] at (9.5,0.5) {};
        \node[cross out, draw, inner sep=2pt, color=OIVermilion] at (9.5,1) {};
        \node[cross out, draw, inner sep=2pt, color=OIVermilion] at (10.5,1) {};
    \end{tikzpicture}
    \caption{The six Weierstrass points on the surface in a hypotetical
    one-cylinder direction of $S_{3,1}$, after affine deformation.}
    \label{fig:no_1cyl_b}
\end{figure}
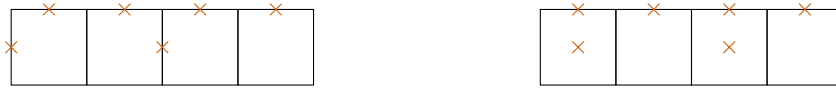

A consequence of this result and \Cref{thm:bad_cyls_subquad_H11} is that the
number of cylinders on the infinite staircase $S_{a,b}^\infty$ is
\emph{subquadratic}, which is our \Cref{bigthm:subquadratically}. To the best of
our knowledge, this is the first such example.

\begin{corollary}\label{cor:SV_staircases_H11}
    Let $a\ge 2$ and $b\ge 0$ two odd integers. Then, the number of cylinders in
    the infinite staircase $S_{a,b}^\infty$ of length at most $L$,
    $N(S_{a,b}^\infty, L)$ is \emph{subquadratic}:  $N(S_{a,b}^\infty, L) =
    o(L^2)$.
\end{corollary}

\begin{figure}[tbh]
    \centering
    \begin{tikzpicture}[scale=1.5, mid arrow/.style={
			   postaction={decorate,decoration={
				   markings,
				   mark=at position .5 with {\arrow{latex}}
		   }}}]
        \foreach \i in {0,...,4}
            \draw (\i,0) rectangle ++(1,1);
        \foreach \i in {0,...,2}
            \draw (4,\i+1) rectangle ++(1,1);
        \draw[OIBluishGreen, mid arrow, very thick] (0,0.5) -- (0.5,1) node[midway,right] {$\delta_s$};
        \draw[OIBluishGreen, mid arrow, very thick] (4.5,0) -- (5,0.5) node[left] {$\delta_s$};
        \draw[OIOrange, mid arrow, very thick] (0.5,0) -- (1.5,1) node[midway,right] {$\delta_l$};
        \draw[OIOrange, mid arrow, very thick] (1.5,0) -- (2.5,1) node[midway,right] {$\delta_l$};
        \draw[OIOrange, mid arrow, very thick] (2.5,0) -- (3.5,1) node[midway,right] {$\delta_l$};
        \draw[OIOrange, mid arrow, very thick] (3.5,0) -- (5,1.5) node[midway,right] {$\delta_l$};
        \draw[OIOrange, mid arrow, very thick] (4,1.5) -- (5,2.5) node[right] {$\delta_l$};
        \draw[OIOrange, mid arrow, very thick] (4,2.5) -- (5,3.5) node[right] {$\delta_l$};
        \draw[OIOrange, mid arrow, very thick] (4,3.5) -- (4.5,4) node[midway,right] {$\delta_l$};

        \draw[OISkyBlue, mid arrow, very thick] (1,0) -- (0,0) node[midway,below] {$\eta_A$};
        \draw[OISkyBlue, mid arrow, very thick] (5,4) -- (4,4) node[midway,above] {$\eta_A$};
        \draw[OIVermilion, mid arrow, very thick] (0,1) -- (1,1) node[midway,above] {$\eta_B$};
        \draw[OIVermilion, mid arrow, very thick] (4,0) -- (5,0) node[midway,below] {$\eta_B$};

    \node[v-bull] at (0,0) {};
    \node[v-bull] at (1,0) {};
    \node[v-bull] at (1,1) {};
    \node[v-bull] at (4,4) {};
    \node[v-bull] at (5,4) {};
    \node[v-bull] at (5,0) {};
    \node[b-bull] at (0,1) {};
    \node[b-bull] at (4,0) {};
    \node[b-bull] at (4,1) {};
    \node[b-bull] at (5,1) {};
    \end{tikzpicture}
    \caption{The two cyliders in the $(1,1)$ direction in the surface $S_{5,3}$, with core curves $\delta_s$ and $\delta_l$.}
    \label{fig:Dehn}
\end{figure}
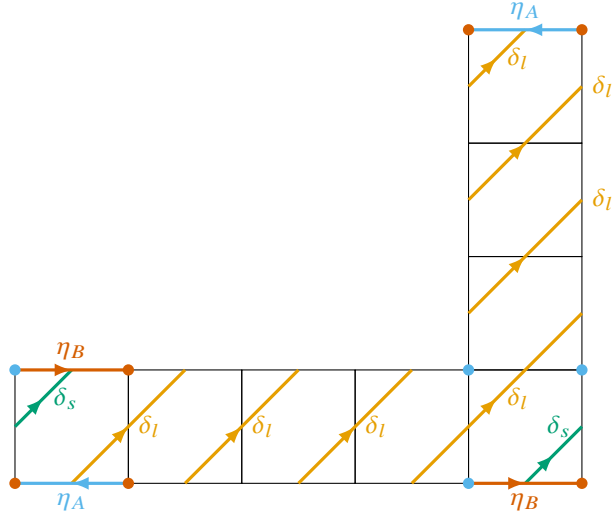

\begin{proof}
    Thanks to \Cref{prop:no_1cyl}, there are no one-cylinder decompositions of
    $S = S_{a,b}$, for $a\ge 2$, $b\ge 0$ odd. If we show that $\eta$ is not
    $\Gamma(S)$-invariant, then this will imply that, even if there were an
    equivariant splitting of the zero holonomy homology as in
    \cref{eq:equivariant_splitting}, $\eta$ does not lie completely in $H^{(0),
    \text{rel}}_1 (S, \Sigma; \ZZ)$. In turn, this implies that $K = H^{(0)}_1
    (S, \Sigma; \ZZ)$ and, by \Cref{lemma:good_are_one-cyl}, good cylinders
    correspond to $1$-cylinders also for the staircases we are considering. 

    In the $(1,1)$ direction, the surface $S$ decomposes into two cylinders, a
    short one, whose core curve we denote $\delta_s$ and a long one, with core
    curve $\delta_l$, see \Cref{fig:Dehn}. Consider the Dehn multitwist $T$ in
    this direction, which correspond to $1$ twist in the long cylinder and
    $a+b-1$ twists in the small one. Recalling that, in general, the Dehn twist
    along $\gamma$ of a curve $\alpha$ is given by
    \[
        T_\gamma (\alpha) = \alpha + \pairing{\gamma,\alpha}\cdot\gamma,
    \]
    we obtain
    \[
        T\eta_A = \eta_A + \delta_l, \qquad \text{and} \qquad T\eta_B = \eta_B - (a+b-1) \delta_s.
    \]
    Hence,
    \[
        \pairing{\eta_A+\eta_B, T(\eta_A + \eta_B)} = \pairing{\eta_A+\eta_B, \eta_A + \delta_l + \eta_B - (a+b-1) \delta_s} = - (a+b-2) \neq 0,
    \]
    which implies that $\eta$ is not $\Gamma(S)$-invariant.
\end{proof}

\section{The remaining Hubert--Schmithüsen
staircase: \texorpdfstring{$S_{2,0}\in\cH(0,0)$}{S(2,0) in H(0,0)}}
\label{sec:H00}

In this last section, we complete the study of the Hubert--Weitze-Schmithüsen
family with its only member of genus one. Namely, the infinite staircase
$\tS_{2,0}$ of \Cref{fig:staircase_cover} on \cpageref{fig:staircase_cover} over
$S = S_{2,0} \in \cH(0,0)$ of \Cref{fig:case00} on \cpageref{fig:case00}, that
is, a torus with two marked points.

\begin{proposition}\label{cor:SV_staircases_H00}
    Let $S = S_{2,0}\in\cH(0,0)$.
    The Siegel--Veech constants associated with the counting problem of cylinders in a one-cylinder decomposition in $S$ is given by
    \[
        c_1(S) = \frac{1}{12\zeta(2)}.
    \]
    The Siegel--Veech constants associated with the counting problem of cylinders in a two-cylinder decomposition in $S$ is given by
    \[
        c_2(S) = \frac{1}{3\zeta(2)} = 4 c_1(S).
    \]
\end{proposition}

\begin{figure}[tbh]
    \centering
    \begin{tikzpicture}[scale=.85]
        \foreach \i in {0,...,1}
            \draw (\i,0) rectangle ++(1,1);

        \draw[very thick] (0,0) -- ++(1,0) node[midway,below] {$A$};
        \draw[very thick] (1,1) -- ++(1,0) node[midway,above] {$A$};
        \draw[very thick] (1,1) -- ++(-1,0) node[midway,above] {$B$};
        \draw[very thick] (2,0) -- ++(-1,0) node[midway,below] {$B$};
        
    \node[v-bull] at (0,0) {};
    \node[v-bull] at (2,0) {};
    \node[v-bull] at (1,1) {};
    \node[b-bull] at (0,1) {};
    \node[b-bull] at (2,1) {};
    \node[b-bull] at (1,0) {};

     \foreach \i in {4,...,5}
            \draw (\i,0) rectangle ++(1,1);

        \draw[very thick] (4,0) -- ++(1,0) node[midway,below] {$A$};
        \draw[very thick] (5,0) -- ++(1,0) node[midway,below] {$B$};
        \draw[very thick] (5,1) -- ++(-1,0) node[midway,above] {$A$};
        \draw[very thick] (6,1) -- ++(-1,0) node[midway,above] {$B$};
        
    \node[v-bull] at (4,0) {};
    \node[v-bull] at (4,1) {};
    \node[v-bull] at (6,0) {};
    \node[b-bull] at (6,1) {};
    \node[b-bull] at (5,0) {};
    \node[b-bull] at (5,1) {};

    \draw (8.5,0.5) rectangle ++(1,1);
    \draw (8.5,-0.5) rectangle ++(1,1);

    \draw[very thick] (8.5,0.5) -- ++(0,1) node[midway,left] {$B$};
    \draw[very thick] (9.5,0.5) -- ++(0,1) node[midway,right] {$B$};
    \draw[very thick] (8.5,0.5) -- ++(0,-1) node[midway,left] {$A$};
    \draw[very thick] (9.5,0.5) -- ++(0,-1) node[midway,right] {$A$};
    
    \node[v-bull] at (8.5,1.5) {};
    \node[v-bull] at (9.5,1.5) {};
    \node[v-bull] at (8.5,-0.5) {};
    \node[b-bull] at (9.5,-0.5) {};
    \node[b-bull] at (8.5,0.5) {};
    \node[b-bull] at (9.5,0.5) {};

    \draw[thick, -{Stealth}, color=OIBlue] (0,-0.25) to [out=220, in=140, looseness=5] (0,1.25);
    \draw[thick, -{Stealth}] (2.25,0.25) to (3.75,0.25);
    \draw[thick, -{Stealth}] (3.75,0.75) to (2.25,0.75);
    \draw[thick, -{Stealth}, color=OIBlue] (6.25,0.75) to (7.75,0.75);
    \draw[thick, -{Stealth}, color=OIBlue] (7.75,0.25) to (6.25,0.25);
    \draw[thick, -{Stealth}] (9.75,1.5) to [out=30, in=330, looseness=3.25] (9.75,-0.5);
    \end{tikzpicture}
\caption{The $\slz$ orbit of $S_{2,0}\in\cH(0,0)$. Blue arrows correspond to the
rotation by $\pi/2$, black arrows to the horizontal shear by $1$.}
\label{fig:Z_20_orbit}
\end{figure}

\begin{proof}
    Since $S\in\cH(0,0)$, there are only two cusps: one associated with
    one-cylinder decompositions and one for two-cylinder ones. The complete
    orbit under $\slz$ is represented in \Cref{fig:Z_20_orbit}. In particular,
    $\#\,\slz\cdot S = 3$. Using \cref{eq:Siegel--Veech_Lelievre} on
    \cpageref{eq:Siegel--Veech_Lelievre} we immediately obtain $c_1(S^{(1)}) =
    1/(6\zeta(2))$. Since $\area(S) = 2$, the claim is proved.

    For the two-cylinder case, \cref{eq:Siegel--Veech_Lelievre} becomes
    \[
        \frac{1}{2\zeta(2)} \frac{n}{\#\,\slz \cdot S} \cdot\cw(C) \left( \frac{1}{w_1^2} + \frac{1}{w_2^2} \right),
    \]
    where $C$ is the two-cylinder cusp and $w_1$ and $w_2$ are the width of the
    two cylinders. In this case, $w_1 = w_2 = 1$ and the cusp width $\cw (C) =
    1$ as well, as can be seen from \Cref{fig:Z_20_orbit} (right), so
    \[
    	c_2(S^{(1)})
	= \frac{1}{2\zeta(2)} \cdot \frac{2}{3} \cdot 1 \cdot (1+1)
	= \frac{2}{3\zeta(2)}.
    \]
    Dividing again by the area of $S$ we are done.
\end{proof}

\begin{remark}
    The same results, and the corresponding computations, can be carried out
    analogously for every square-tiled surface in $\cH(0,0)$, and we do not
    give the details. We only remark that, in this case, $K = \QQ\eta$, so that
    $K_\eta = K$, and good cylinders, that is, cylinders in one-cylinder
    directions, are exactly the cylinders that lift to $\tS$; in particular,
    there are no bad cylinders.
\end{remark}

Let $q$ be a prime, and consider, as before, the degree $q$ cyclic cover $\qS$,
defined by $\eta$. Since $\eta$ has to be holonomy zero, it is necessarily a
multiple of $\eta_A + \eta_B$, where $\eta_A$ and $\eta_B$ are as in
\Cref{fig:Z20_basis_homology}. We restrict ourselves with the case $\eta =
\eta_A + \eta_B$.

By \Cref{fig:Z_20_orbit}, the Veech group $\Gamma(S)$ of $S$ is generated by 
\[
    R = \begin{pmatrix}
        0 & -1 \\
        1 & 0
    \end{pmatrix}
    \qquad \text{and} \qquad
    T_h = \begin{pmatrix}
        1 & 2 \\
        0 & 1
    \end{pmatrix}.
\]
This is a subgroup of index $3$ inside $\slz$, which has two cusps. The cusp
directions correspond to $\QQ \cup \{\infty\}$. One can see that the two
cusps correspond to the orbit of $1$ and $0$ under the group. The former is the
one formed by rational numbers with odd numerator and denominator, and gives
rise to the two-cylinders decompositions. The latter corresponds to rational
numbers in which the numerator or the denominator is even, and corresponds to
the one-cylinder decompositions, i.e., the good cylinders. We now show that all
two-cylinders lift to $q$-strips in the cover.

\begin{figure}[tbh]
    \centering
    \begin{tikzpicture}[scale=1.5, mid arrow/.style={
			   postaction={decorate,decoration={
				   markings,
				   mark=at position .5 with {\arrow{latex}}
		   }}}]
        \foreach \i in {0,...,1}
            \draw (\i,0) rectangle ++(1,1);

        \draw[OISkyBlue, mid arrow, very thick] (1,0) -- (0,0) node[midway,below] {$\eta_A$};
        \draw[OISkyBlue, mid arrow, very thick] (2,1) -- (1,1) node[midway,above] {$\eta_A$};
        \draw[OIVermilion, mid arrow, very thick] (0,1) -- (1,1) node[midway,above] {$\eta_B$};
        \draw[OIVermilion, mid arrow, very thick] (1,0) -- (2,0) node[midway,below] {$\eta_B$};
        \draw[OIOrange, mid arrow, very thick] (0,0) -- (0,1) node[midway,left] {$v$};
        \draw[OIOrange, mid arrow, very thick] (2,0) -- (2,1) node[midway,right] {$v$};
        \draw[mid arrow, very thick] (1,1) -- (1,0) node[midway,right] {$w$};

    \node[v-bull] at (0,0) {};
    \node[v-bull] at (2,0) {};
    \node[v-bull] at (1,1) {};
    \node[b-bull] at (0,1) {};
    \node[b-bull] at (2,1) {};
    \node[b-bull] at (1,0) {};
    \end{tikzpicture}
    \caption{The basis of $H_1(S, \Sigma,\ZZ)$.}
    \label{fig:Z20_basis_homology}
\end{figure}

\begin{proposition}\label{prop:no_bad_cyls}
    Let $p/q$ be a rational number with odd numerator and denominator. Then, the
    core curve of each one of the cylinders in direction $p/q$ has intersection
    $\pm 1$ with $\eta$. In particular, the cylinders in direction $p/q$ lift to
    $q$-strips in $\qS$.
\end{proposition}

\begin{proof}
    Let $\gamma$ be the core curve of a cylinder in direction $p/q$. We will
    show that $\langle \gamma, \eta \rangle = \pm 1$. Since the Veech group
    $\Gamma(S)$ acts transitively on two-cylinders directions (there is only one
    such cusp), it is enough to show that $\Gamma(S)$ acts trivially on $\eta$
    and then check that the intersection of $\eta$ with the core curves in
    direction $1$ is equal to $\pm 1$.

    The horizontal Dehn twist clearly stabilizes $\eta$, so let us focus on the
    rotation $R$. Consider the basis of $H_1 (S, \Sigma, \ZZ)$ given by $\{v,
    \eta_A, \eta_B\}$, see \Cref{fig:Z20_basis_homology}. Then, $w = -v -
    \eta_A - \eta_B$. The induced action on $H_1 (S, \Sigma, \ZZ)$ is given
    by
    \[
        R_* \eta_A = -v, \qquad\text{and}\qquad R_* \eta_B = -w.
    \]
    Then, $R_* \eta = - (v+w) = \eta_A + \eta_B = \eta$. By a direct
    computation, the intersection between any of the core curves in direction
    $(1,1)$ and $\eta$ is $1$ and we are done.
\end{proof}

\begin{remark}
    The previous result also follows from the fact that $K = \ker\hol = \langle
    \eta \rangle$, which is $\Gamma(S)$-invariant.
\end{remark}

We conclude by showing that the analogue of \Cref{bigthm:large_genus} holds when
the base is the surface $S_{2,0}$.

\begin{proposition}
    Let $q$ be a prime number and $\qS$ a $\ZZ_q$-cover of $S=S_{2,0}$ defined
    by $\eta$. Then
    \[
        \frac{1}{q} c(\qS) \xrightarrow{q\to\infty} c(\tS) = \frac{1}{12\zeta(2)},
    \]
    where $c(\qS)$ is the Siegel--Veech constant associated to $N(\qS,L)$.
\end{proposition}

\begin{proof}
    We begin by recalling that the area of $\qS$ is $q$ times the area of $S$.
    Hence, one-cylinder decompositions in $S$ lift to $q$ isometric copies of
    the base cylinder, whereas two cylinders decompositions lift to $q$ times
    longer cylinders. By the above \Cref{prop:no_bad_cyls}, there are no bad
    cylinders in $S$, so, using the notation of the proof of
    \Cref{bigthm:large_genus}, the set $A$ of not good cylinders coincides with
    the set $T_q = T$ of $q$-strips (which does not depend on $q$). Then,
    \cref{eq:q-Siegel--Veech} on \cpageref{eq:q-Siegel--Veech} becomes
    \[
        c(\qS) = qc(G) + \frac{1}{q^2} c(A).
    \]
    Then
    \[
        \lim_{\substack{q\to\infty \\ q \text{ prime}}} \frac{1}{q} c(\qS) = c(G) = c(\tS) = \frac{1}{12\zeta(2)}.\qedhere
    \]
\end{proof}

\printbibliography
\end{document}

\typeout{get arXiv to do 4 passes: Label(s) may have changed. Rerun}